\documentclass[oneside, 11pt, reqno]{amsart}

\usepackage[dvipsnames]{xcolor}
\usepackage[utf8]{inputenc}
\usepackage{amsmath,amsthm,amssymb}
\usepackage[shortlabels]{enumitem}
\usepackage{etoolbox}
\usepackage{geometry}
\usepackage[numbers]{natbib}
\usepackage{float}
\usepackage{mathrsfs}
\usepackage{mathtools}

\usepackage{subcaption}
\usepackage{hyperref}
\hypersetup{linktoc=all,colorlinks=true,allcolors=Blue,breaklinks=true}
\usepackage[nosort,capitalize,noabbrev]{cleveref}
\usepackage{xspace}
\usepackage{tikz}
\usepackage{tikz-cd}

\newcommand{\bibsortkey}[1]{}
\usetikzlibrary{calc}

\newtheorem{theorem}{Theorem}[section]
\newtheorem{lemma}[theorem]{Lemma}
\newtheorem{proposition}[theorem]{Proposition}
\newtheorem{corollary}[theorem]{Corollary}
\newtheorem{claim}[theorem]{Claim}

\theoremstyle{definition}
\newtheorem{definition}[theorem]{Definition}
\newtheorem{remark}[theorem]{Remark}
\newtheorem{example}[theorem]{Example}

\DeclareMathOperator{\core}{core}
\DeclareMathOperator{\reg}{reg}
\DeclareMathOperator{\intr}{int}

\DeclareMathOperator{\cl}{cl}

\newcommand{\ds}{\mathord{\rho}}
\newcommand{\dcore}{\core_d}

\newcommand{\upset}{\mathord{\uparrow}}
\newcommand{\downset}{\mathord{\downarrow}}
\newcommand{\clopup}{{\sf ClopUp}}

\newcommand{\clopsup}{{\sf ClopSUp}}

\newcommand{\cat}[1]{{\sf #1}\xspace}
\newcommand{\functor}[1]{\mathscr #1}
\newcommand{\DLat}{\cat{DLat}}
\newcommand{\Pries}{\cat{Pries}}
\newcommand{\LPries}{\cat{LPries}}
\newcommand{\SL}{\cat{SLPries}}

\newcommand{\AlgL}{\cat{AlgL}}
\newcommand{\AriL}{\cat{AriL}}

\newcommand{\Frm}{\cat{Frm}}
\newcommand{\SFrm}{\cat{SFrm}}
\newcommand{\AlgFrm}{\cat{AlgFrm}}
\newcommand{\AriFrm}{\cat{AriFrm}}

\colorlet{red}{OrangeRed}

\crefname{claim}{Claim}{Claims}

\setlist[enumerate,1]{label={\upshape(\arabic*)}}

\makeatletter
\edef\plabelformat{\string#2\string#1\string#3}
\edef\plabelrangeformat{\string#3\string#1\string#4--\string#5\string#2\string#6}
\newcommand{\plabel}[1]{\label{#1}
\immediate\write\@auxout{\noexpand\crefformat{#1}{\noexpand\cref{#1}\plabelformat}
\noexpand\crefmultiformat{#1}{\noexpand\cref{#1}\plabelformat}{,\plabelformat}{,\plabelformat}{,\plabelformat}
\noexpand\crefrangeformat{#1}{\noexpand\cref{#1}\plabelrangeformat}}}
\makeatother

\title{Maximal $d$-spectra via Priestley duality}

\makeatletter
\@namedef{subjclassname@2020}{
  \textup{2020} Mathematics Subject Classification}
\patchcmd{\@setaddresses}{\indent}{\noindent}{}{}
\patchcmd{\@setaddresses}{\indent}{\noindent}{}{}
\patchcmd{\@setaddresses}{\indent}{\noindent}{}{}
\patchcmd{\@setaddresses}{\indent}{\noindent}{}{}
\makeatother
\author{G.~Bezhanishvili, P.~Bhattacharjee, S.\ D.~Melzer}



\begin{document}

\begin{abstract}
	We use Priestley duality as a new tool to study maximal $d$-spectra of arithmetic frames, both with and without units.
	We pay special attention to when the maximal $d$-spectrum is compact or Hausdorff. Various necessary and sufficient
	conditions are given, including a construction of an arithmetic frame with a unit whose maximal $d$-spectrum is not
	Hausdorff, thus resolving an open problem in the literature. 
\end{abstract}

\begingroup
\def\uppercasenonmath#1{}
\let\MakeUppercase\relax
\maketitle
\endgroup

\tableofcontents

\section{Introduction}

The space of maximal $d$-ideals of an archimedean Riesz space with a weak order unit, equipped with the hull-kernel
topology, has been well studied and is known to be a compact Hausdorff space (see, e.g., \cite{HuijsmansPagter1980}). 
Motivated by this, Martinez and Zenk \cite{MartinezZenk2003} initiated the study of $d$-elements in an arbitrary
arithmetic frame. These elements, denoted $dL$ or $L_d$, form a sublocale of the arithmetic frame $L$. The corresponding
nucleus was coined the \emph{$d$-nucleus}. The frame $L_d$ and the spectrum $\max L_d$ of maximal $d$-elements were
further studied by various authors (see, e.g.,
\cite{DubeIghedo2014,DubeSithole2019,Bhattacharjee2018,Bhattacharjee2024}). It is known that $\max L_d$ is a compact
$T_1$-space provided $L$ has a unit. In \cite{Bhattacharjee2018}, it was left open whether $\max L_d$ is Hausdorff.
Although some characterizations of the Hausdorff separation for $\max L_d$ were recently established in
\cite{Bhattacharjee2024}, the question remained open. Our aim is to answer it in the negative.

Our main tool is Priestley duality for frames. Priestley originally developed her duality for bounded distributive
lattices \cite{Priestley1970,Priestley1972}. It was restricted to the category of frames by Pultr and Sichler
\cite{PultrSichler1988}, and later the Priestley spaces of arithmetic frames were characterized in
\cite{BezhanishviliMelzer2023}. Building on this work, we describe the subset $N_d$ of the Priestley space $X$ of an
arithmetic frame $L$ corresponding to the $d$-nucleus on $L$. We also describe the subset $Y_d$ of $N_d$ corresponding
to the spectrum $pt(L_d)$ of points of the sublocale $L_d$. We show that the minimum of $Y_d$ is in a one-to-one
correspondence with the maximal $d$-elements of $L$, thus yielding a homeomorphism between $\min Y_d$ and $\max L_d$.
This allows us to study $\max L_d$ in the language of Priestley spaces.
Our main results include the description of the nucleus on $L$ whose fixpoints are the frame of opens of $\min Y_d$,
the characterization of the soberification of $\min Y_d$, and the construction of an arithmetic frame $L$ with a unit
such that $\min Y_d$ is not Hausdorff. This yields that $\max L_d$ is not Hausdorff, thus resolving the open question
mentioned above. We also give a necessary and sufficient condition for $\max L_d$ to be Hausdorff. This characterization
remains valid even if $L$ doesn't have a unit, provided $\max L_d$ is locally compact. We also investigate the
compactness of $\min Y_d$, and hence of $\max L_d$, in comparison with the existence of a unit. Our approach raises
new open questions and highlights the need for further study of maximal $d$-spectra of arithmetic frames (see the end of
the paper).

The paper is structured as follows. In \cref{section 2}, we recall Priestley duality for arithmetic frames, along with
useful definitions and results from the literature. In \cref{section 3}, we revisit the relationship between nuclei and
sublocales, as well as their description in the language of Priestley spaces. 
In particular, we show how to use Priestley duality to give alternative proofs of two existing results in the literature; Johonstone's lemma that each Scott open filter arises as the dense elements of a nucleus and the Isbell Density Theorem.

In \cref{section 4}, we characterize inductive nuclei in the language of Priestley spaces, as well as provide the dual
description of the $d$-nucleus on an arithmetic frame $L$. In \cref{new section 5}, we study the maximum of the localic
part of the Priestley dual of $L$, which yields a new characterization of when the sublocale $L_d$ is regular. 
In \cref{section 5}, we delve into the investigation of the spectrum $\max L_d$ of maximal $d$-elements of an arithmetic
frame. We establish a homeomorphism between $\max L_d$ and $\min Y_d$, thus giving us a new tool to study the maximal
$d$-spectrum of $L$. We show that the frame of open sets of $\min Y_d$ can be realized as a sublocale of $L$ and
describe the corresponding nuclear subset of $X$. In addition, we prove that the localic part of this nuclear subset is
the soberification of $\min Y_d$.

In \cref{sec: compactness}, we study the topological properties of $\min Y_d$ in comparison to what is known about
$\max L_d$. We describe compact subsets of $\min Y_d$, which allows us to characterize when an arithmetic frame has a
unit using Priestley duality. This in particular yields that in the presence of a unit, $\min Y_d$ is a compact space. 

Finally, in \cref{Section 7}, we explore the Hausdorff separation for the space $\min Y_d$. An example of an arithmetic
frame $L$ with a unit is constructed such that $\max L_d$ and hence $\min Y_d$ is not Hausdorff, thus resolving an open
question from \cite{Bhattacharjee2018}. Furthermore, we give a characterization of when $\min Y_d$ is Hausdorff, which
generalizes to arithmetic frames without units, provided $\min Y_d$ is locally compact. The article concludes with
several open questions that the authors are looking into to further the study of $\min Y_d$.

\section{Priestley duality for frames} \label{section 2}
A \emph{frame} is a complete lattice $L$ such that
\[
	a \wedge \bigvee S = \bigvee \{a \wedge s \mid s \in S\}
\]
for all $a \in L$ and $S \subseteq L$. A \emph{frame homomorphism} is a map between frames that preserves finite meets
and arbitrary joins. Let \Frm be the category of frames and frame homomorphisms. A frame $L$ is \emph{spatial} if it is
isomorphic to the frame of opens of a topological space (equivalently, completely prime filters separate elements of
$L$). Let \SFrm be the full subcategory of \Frm consisting of spatial frames.

An element $a \in L$ is {\em compact} if for each $S \subseteq L$, from $a \leq \bigvee S$ it follows that
$a \leq \bigvee T$ for some finite $T \subseteq S$. Let $K(L)$ be the set of compact elements of $L$. Then $L$ is
\emph{algebraic} provided
\[
	a = \bigvee \{b \in K(L) \mid b \leq a\}
\]
for each $a \in L$. It is well known (see, e.g., \cite[Rem.~3.4]{MartinezZenk2003}) that every algebraic frame is
spatial. A frame homomorphism $h : L \to M$ is \emph{coherent} if $h[K(L)] \subseteq K(M)$. Let \AlgFrm be the category
of algebraic frames and coherent frame homomorphisms between them. An algebraic frame $L$ is \emph{arithmetic} if $K(L)$
is closed under binary meets. Let \AriFrm be the full subcategory of \AlgFrm consisting of arithmetic frames.

A space is \emph{zero-dimensional} if it has a basis of clopen sets. A \emph{Stone space} is a zero-dimensional,
compact, Hausdorff space. A \emph{Priestley space} is a pair $(X,\le)$ such that $X$ is a Stone space and $\le$ is a
partial order on $X$ such that the \emph{Priestley separation} holds:
\begin{center}
	If $x \nleq y$ then there is a clopen upset $U$ of $X$ containing $x$ and missing $y$.    
\end{center}
A \emph{Priestley morphism} is a continuous order-preserving map between Priestley spaces. Let \Pries be the category of
Priestley spaces and Priestley morphisms. Let \DLat be the category of bounded distributive lattice and bounded lattice
homomorphisms.

\begin{theorem}[Priestley duality]
	\DLat and \Pries are dually equivalent.
\end{theorem}

\begin{remark}
	The functors $\functor X : \DLat \to \Pries$ and $\functor{D} : \Pries \to \DLat$ establishing Priestley duality are
	described as follows.
	\begin{itemize}
		\item The \emph{Priestley space} of a bounded distributive lattice $D$ is the set $X_D$ of prime filters of $D$
		ordered by inclusion and topologized by the basis $\{\varphi(a) \setminus \varphi(b) \mid a,b \in D\}$, where
		$\varphi$ is the \emph{Stone map} defined by $\varphi(a) = \{x \in X_D \mid a \in x\}$ for all $a \in D$. The
		functor $\functor X$ assigns to each $D \in \DLat$ its Priestley space $X_D$, and to each bounded lattice
		homomorphism $h : D \to E$ the Priestley morphism $h^{-1} : X_E \to X_D$.
		\item The functor $\functor D$ assigns to each Priestley space $X$ the bounded distributive lattice $\clopup(X)$ of
		clopen upsets of $X$ and to each Priestley morphism $f : X \to Y$ the bounded lattice homomorphism
		$f^{-1} : \clopup(Y) \to \clopup(X)$.
	\end{itemize}
\end{remark}

\begin{definition}
	\begin{enumerate}
		\item[]
		\item An \emph{L-space} (\emph{localic space}) is a Priestley space such that the closure of each open upset is an
		open upset.
		\item An \emph{L-morphism} is a Priestley morphism $f : X \to Y$ between L-spaces satisfying
		$f^{-1}(\cl U) = \cl f^{-1}(U)$ for each open upset $U$ of $Y$.
		\item \LPries is the category of L-spaces and L-morphisms.
	\end{enumerate}
\end{definition}

\begin{remark}\label{esakia}
	Recall that frames are precisely complete Heyting algebras. By Esakia duality \cite{Esakia1974}, the Priestley spaces
	of Heyting algebras (Esakia spaces) are those with the property that the closure of each upset is an upset. Moreover,
	if $X$ is the Priestley space of an Heyting algebra $L$, then $L$ is complete iff $X$ is
	\emph{extremally order-disconnected}; that is, the closure of each open upset is open. Thus, L-spaces are precisely
	extremally order-disconnected Esakia spaces.
\end{remark}

\begin{theorem}[Pultr-Sichler duality]
	\Frm and \LPries are dually equivalent.
\end{theorem}

Throughout we will use the following well-known facts (see, e.g.,
\cite{Priestley1984,Esakia2019,BezhanishviliBezhanishvili2008}).
\begin{lemma}
    Let $X$ be a Priestley space.
\begin{enumerate}[ref=\thelemma(\arabic*)]
    \item If $F \subseteq X$ is closed, then so are $\upset F$ and $\downset F$.
    \label[lemma]{prelim fact-up down closed}
    \item If $F \subseteq X$ is a closed upset, then it is an intersection of clopen upsets.
    \label[lemma]{prelim fact-intersections}
    \item If $F \subseteq X$ is closed, then for each $x \in F$ there are $m \in \min F$ and $n \in \max F$ such that
    $m \leq x \leq n$, where $\min F$ and $\max F$ denote the sets of minimal and maximal points of $F$, respectively.
    \label[lemma]{prelim fact-min max exist}
\end{enumerate}
If in addition $X$ is an L-space, then we have\textup{:}
\begin{enumerate}[ref=\thelemma(\arabic*),resume]
    \item If $F \subseteq X$ is closed, then so is $\max F$.\label[lemma]{prelim fact-max closed}
    \item If $U \subseteq X$ is clopen, then so is $\downset U$.\label[lemma]{prelim fact-clopen}
    \item For $\{ U_i\} \subseteq \clopup(X)$,\label[lemma]{prelim fact-join and meet}
		\[
		    \bigvee U_i = \cl \bigcup U_i \quad \mbox{and}\quad 
		    \bigwedge U_i = X \setminus \downset(X \setminus \intr \bigcap U_i). 
		\]
\end{enumerate}
\end{lemma}

We next describe the Priestley spaces of spatial, algebraic, and arithmetic frames.

\begin{definition}
	For an L-space $X$,
	\begin{enumerate}
		\item $y \in X$ is a \emph{localic point} if $\downset y$ is clopen;
		\item the set $Y$ of localic points of $X$ is the \emph{localic part} of $X$;
		\item $X$ is an \emph{SL-space} (\emph{spatial} L-space) if $Y$ is dense in $X$;
		\item \SL is the full subcategory of \LPries consisting of SL-spaces.
	\end{enumerate}
\end{definition}

\begin{theorem}[see, e.g., {\cite[Cor.~4.10]{BezhanishviliMelzer2022b}}]
	\SFrm and \SL are dually equivalent.
\end{theorem}

\begin{remark} \label{topology Y}
	Let $X$ be the Priestley space of a frame $L$, and $Y$ the localic part of $X$. We topologize $Y$ by
	$\{U \cap Y \mid U \in \clopup(X)\}$. Then $Y$ is precisely the space of points of $L$ (see, e.g,
	\cite[Prop~5.1]{AvilaBezhanishviliMorandiZaldivar2020}).
\end{remark}

\begin{definition} 
	Let $X$ be an L-space and $Y$ its localic part.
	\begin{enumerate}
		\item A closed upset $F \subseteq X$ is a \emph{Scott upset} if $\min F \subseteq Y$.
		\item $\clopsup(X)$ is the collection of clopen Scott upsets of $X$.
    \end{enumerate}
\end{definition}

Scott upsets play an important role in Priestley spaces of frames as they correspond to Scott open filters on the frame
side:
\begin{theorem}[{\cite[Thm.~5.6]{BezhanishviliMelzer2022}}]
  \label{thm: scott upsets}
  Let $L$ be a frame and $X$ its L-space. The poset of Scott open filters of $L$ is dually isomorphic to the poset of
  Scott upsets of $X$.
\end{theorem}

\begin{definition}
  Let $X$ be an L-space and $Y$ its localic part.
  \begin{enumerate}
		\item The \emph{core} of $U \in \clopup(X)$ is $\core U = \bigcup\{V \in \clopsup(X) \mid V \subseteq U\}$.
		\item $X$ is an \emph{algebraic L-space} if $\core U$ is dense in $U$ for each $U \in \clopup(X)$.
		\item An L-morphism $f : X_1 \to X_2$ is \emph{coherent} if $f^{-1}(\core U) \subseteq \core f^{-1}(U)$ for each
		$U \in \clopup(X_2)$.
		\item \AlgL is the category of algebraic L-spaces and coherent L-morphisms between them.
	\end{enumerate}
\end{definition}

\begin{theorem}[{\cite[Thm.4.9]{BezhanishviliMelzer2023}}]
	\AlgFrm and \AlgL are dually equivalent.
\end{theorem}
Consequently, every algebraic L-space is an SL-space.

\begin{remark} \label{rem: 2.8}
	Let $X$ be an L-space and $U \in \clopup(X)$. Then $\core U = U$ iff $U$ is a Scott upset, which by
	\cite[Cor.~5.4]{BezhanishviliMelzer2022} is equivalent to $U$ being a compact element of $\clopup(X)$. Thus,
	$\clopsup(X) = K(\clopup(X))$.
\end{remark}

\begin{definition}
	\begin{enumerate}[ref=\thedefinition(\arabic*)]
		\item[]
		\item\label[definition]{def arithemitc} An \emph{arithmetic L-space} is an algebraic L-space $X$ such that
		\[
			\core U \cap \core V = \core (U \cap V)
		\]
		for all $U,V \in \clopup(X)$. Equivalently (see, e.g., \cite[Lem.~5.2]{BezhanishviliMelzer2023}),
		\[
			U,V \in \clopsup(X) \implies U \cap V \in \clopsup(X).
		\]
		\item \AriL is the full subcategory of \AlgL consisting of arithmetic L-spaces.
	\end{enumerate}
\end{definition}

\begin{theorem}[{\cite[Thm~5.5]{BezhanishviliMelzer2023}}]\label{thm: another duality}
	\AriFrm and \AriL are dually equivalent.
\end{theorem}

\section{Priestley duals for nuclei and sublocales} \label{section 3}

In this section, we recall the relationship between nuclei and sublocales of a frame $L$, as well as their dual
characterization as nuclear subsets of the Priestley space of $L$. Moreover, we give a dual characterization of the
admissible filter of a nucleus on $L$, which yields an alternative proof of \cite[Lem.~3.4(ii)]{Johnstone1985}. Furthermore, we provide an alternate proof of Isbell's Density Theorem, which utilizes
Priestley duality.

\begin{definition}[see, e.g., {\cite[p.~48]{Johnstone1982}}]
	A \emph{nucleus} on a frame $L$ is a map $j : L \to L$ satisfying
	\begin{enumerate}
		\item $a \leq ja $
		\item $jja \leq ja$
		\item $ja \wedge jb = j(a \wedge b)$
	\end{enumerate}
	for all $a, b \in L$.
\end{definition}
Let $N(L)$ be the set of nuclei on $L$. We order $N(L)$ pointwise, i.e., $j \leq k$ iff $j(a) \leq k(a)$ for all
$a \in L$. With this order, it is well known that  $N(L)$ is a frame (see, e.g., \cite[p.~51]{Johnstone1982}).

\begin{definition}[see, e.g., {\cite[p.~26]{PicadoPultr2012}}] \label{def: sublocale}
	Let $L$ be a frame. We call $S \subseteq L$ a \emph{sublocale} of $L$ provided $S$ is closed under arbitrary meets and
	$a \to s \in S$ for all $a \in L$ and $s \in S$.
\end{definition}

Let $S(L)$ be the set of sublocales of $L$. We order $S(L)$ by inclusion. Then $S(L)$ is dually isomorphic to $N(L)$
(see, e.g., \cite[Sect.~III-5]{PicadoPultr2012}). The dual isomorphism associates with each nucleus $j$, the sublocale
$S_j \coloneqq j[L]\in S(L)$; and with each sublocale $S$, the nucleus $j_S$ given by
$j_S(a) = \bigwedge \{s \in S \mid a \leq s\}$.
Nuclei, and hence sublocales, dually correspond to nuclear subsets of L-spaces introduced in \cite{PultrSichler2000}
(see also \cite{BezhGhilardi2007,AvilaBezhanishviliMorandiZaldivar2020}):
\begin{definition}
	Let $X$ be an L-space.
	\begin{enumerate}
		\item A subset $N$ of $X$ is a \emph{nuclear subset} provided $N$ is closed and $\downset(U \cap N)$ is clopen for
		each clopen subset $U$ of $X$.
		\item Let $N(X)$ be the set of nuclear subsets of $X$ ordered by inclusion.
	\end{enumerate}
\end{definition}

\begin{theorem}[{\cite[Thm.~30]{BezhGhilardi2007}}]
	Let $L$ be a frame and $X$ its Priestley space. Then $N(L)$ is dually isomorphic to $N(X)$.
\end{theorem}

\begin{remark}\label{NL = NX remark}
	The dual isomorphism of the previous theorem is established as follows: with each $j \in N(L)$ we associate the
	nuclear subset 
	\[
	N_j \coloneqq \{x \in X \mid j^{-1}[x] = x\} \in N(X),
	\] 
	and with each $N \in N(X)$ the nucleus
	$j_N \in N(\clopup(X))$ given by ${j_N\,U = X \setminus \downset (N \setminus U)}$. Then
	$j\coloneqq\varphi^{-1} \circ j_N \circ \varphi$ is the corresponding nucleus on $L$. 
 \end{remark}

To simplify notation, we identify $N(L)$ with $N(\clopup(X))$. Thus, each $j\in N(L)$ is identified with $j_{N_j}$, and
hence, for $U \in \clopup(X)$, we have $j\,U = X \setminus \downset (N_j \setminus U)$.

\begin{corollary}[{\cite[p.~229]{PultrSichler2000}}]\label{S(L)=N(X)}
	Let $L$ be a frame and $X$ its Priestley space. Then $S(L)$ is isomorphic to $N(X)$.
\end{corollary}

\begin{remark} \plabel{rem nuclear}
Let $L$ be a frame and $X$ its Priestley space. 
\begin{enumerate}
	\item\label[rem nuclear]{remark: Priestley space of sublocale}If $j \in N(L)$, then $N_j$ seen as a subspace of $X$ is
	order-homeomorphic to the Priestley space of the sublocale $S_j$ of $L$ (see, e.g., \cite[Lem.~25]{BezhGhilardi2007}).
	\item Localic points of $X$ also known as \emph{nuclear points} (see, e.g.,
	\cite[Def.~4.1]{AvilaBezhanishviliMorandiZaldivar2020}). This is because $y \in X$ is localic iff $\{y\}$ is a nuclear
	subset.
	\item\label[rem nuclear]{join nuclear}By \cite[Lem.~4.8]{AvilaBezhanishviliMorandiZaldivar2020}, the join in $N(X)$ is
	calculated by
	\[
		\bigvee N_i = \cl \bigcup N_i
	\]
	for $\{N_i\} \subseteq N(X)$.
	\item By \ref{join nuclear}, $\cl Z$ is a nuclear subset of $X$ for every subset $Z \subseteq Y$ of the localic part.
	\label[rem nuclear]{subsets are nuclear}
\end{enumerate}
\end{remark}

Let $j$ be a nucleus on $L$. An element
$a \in L$ is called \emph{$j$-dense} provided $ja = 1$.
Let $F_j$ be the set of all $j$-dense elements of $L$. It is well known and straightforward to verify that $F_j$ is a
filter of $L$. 

\begin{definition}[see, e.g., \cite{Simmons2006}]
	A filter $F$ of a frame $L$ is called \emph{admissible} if it is of the form $F = F_j$ for some $j \in N(L)$.
\end{definition}
\begin{remark}
  In \cite{MartinezMcGovern2009} admissible filters are called \emph{smooth}. In \cite[Thm.~25.5]{Wilson1994} it is
  shown that a filter is admissible iff it is \emph{free}, a concept that is now known as \emph{strongly exact} (see,
  e.g., \cite{MoshierPultrSuarez2020}).
\end{remark}

Let $X$ be the Priestley space of $L$. Recalling the well-known correspondence between filters of $L$
and closed upsets of $X$ (see \cite[p.~54]{Priestley1984} or \cite[Cor.~6.3]{BezhanishviliGabelaiaKurz2010}), let
${H_j \coloneqq \bigcap \varphi[F_j]}$ be the closed upset of $X$ corresponding to $F_j$.
To describe the relationship between $N_j$ and $H_j$,
we require the following lemma.

\begin{lemma}\plabel{lem: Nj}
	Let $L$ be a frame, $X$ its Priestley space, $j \in N(L)$, and $a,b \in L$.
	\begin{enumerate}
		\item\label[lem: Nj]{Nj-1} $\varphi(a) \cap N_j = \varphi(ja) \cap N_j$.
		\item\label[lem: Nj]{Nj-2} $x \in \varphi(ja)$ iff $\upset x \cap N_j \subseteq \varphi(a) \cap N_j$.
		\item\label[lem: Nj]{Nj-3} $ja \leq jb$ iff $\varphi(a) \cap N_j \subseteq \varphi(b) \cap N_j$.
		\item\label[lem: Nj]{Nj-4} $ja = jb$ iff $\varphi(a) \cap N_j = \varphi(b) \cap N_j$.
		\item\label[lem: Nj]{j=1 iff contains N_j} $a$ is $j$-dense iff $N_j \subseteq \varphi(a)$.
	\end{enumerate}
\end{lemma}
\begin{proof}
	\ref{Nj-1} The left-to-right inclusion is clear because $a\le ja$. For the right-to-left inclusion, suppose
	$x \in \varphi(ja) \cap N_j$. Then $ja \in x$, so $a \in j^{-1}[x]$. But $j^{-1}[x] = x$ since $x \in N_j$.
	Consequently, $x \in \varphi(a) \cap N_j$.

	\ref{Nj-2} Let $x \in X$. We have
	\begin{align*}
		x \in \varphi(ja)
		&\iff x \in X \setminus \downset (N_j \setminus \varphi(a)) \\
		&\iff x \notin \downset (N_j \setminus \varphi(a))\\
		&\iff \upset x \cap (N_j \setminus \varphi(a)) = \varnothing\\
		&\iff \upset x \cap N_j \subseteq \varphi(a)\cap N_j.
	\end{align*}

	\ref{Nj-3} Suppose $ja \leq jb$. Then $\varphi(ja) \subseteq \varphi(jb)$. Consequently,
	\[
		\varphi(ja) \cap N_j \subseteq \varphi(jb) \cap N_j,
	\]
	and so $\varphi(a) \cap N_j \subseteq \varphi(b) \cap N_j$ by
	\ref{Nj-1}. Conversely, suppose $\varphi(a) \cap N_j \subseteq \varphi(b) \cap N_j$. It suffices to show that
	$\varphi(ja) \subseteq \varphi(jb)$. Let $x \in \varphi(ja)$. Then $\upset x \cap N_j \subseteq \varphi(a) \cap N_j$
	by \ref{Nj-2}. Therefore, $\upset x \cap N_j \subseteq \varphi(b) \cap N_j$ by assumption. Thus, $x \in \varphi(jb)$
	by \ref{Nj-2}.

	\ref{Nj-4} This follows from \ref{Nj-3}.

	\ref{j=1 iff contains N_j} Suppose $ja = 1$. Then $\varphi(ja) = X$. Therefore,  by \ref{Nj-1},
	\[
		N_j \cap \varphi(a) = N_j \cap \varphi(ja) = N_j.
	\]
	Thus, $N_j \subseteq \varphi(a)$. Conversely, suppose
	$N_j \subseteq \varphi(a)$. Then $\varphi(a) \cap N_j = N_j = \varphi(1) \cap N_j$. Therefore,  by \ref{Nj-4},
	$ja = j1 = 1$.
\end{proof}

\begin{theorem}\label{upset N_j = F_j}
	 Let $L$ be a frame, $X$ its Priestley space, and $j \in N(L)$. Then $H_j = \upset N_j$.
\end{theorem}

\begin{proof}
	Since both $H_j$ and $\upset N_j$ are closed upsets, and hence intersections of clopen upsets (see
	\cref{prelim fact-intersections}), it is sufficient to show that $H_j \subseteq \varphi(a)$ iff
	$\upset N_j \subseteq \varphi(a)$ for each $a\in L$. We have
	\[
		H_j = \bigcap_{b \in F_j} \varphi(b) \subseteq \varphi(a) \iff a \in F_j \iff ja = 1 \iff N_j \subseteq \varphi(a)
		\iff \upset N_j \subseteq \varphi(a),
	\]
	where the first equivalence follows from compactness and the second to last equivalence from
	\cref{j=1 iff contains N_j}.
\end{proof}

The following result about Scott open filters was established in \cite[Lem.~3.4(ii)]{Johnstone1985} using transfinite
induction. Since then various alternative proofs have been obtained (see, e.g., \cite[Sec.~5.1]{JaklSuarez2025} and the references therein). We utilize Priestley duality to provide a simpler proof. 
\begin{corollary}
	Every Scott open filter of a frame is admissible.
\end{corollary}
\begin{proof}
  Let $L$ be a frame and $X$ its L-space. By \cref{thm: scott upsets}, Scott open filters correspond to Scott upsets;
  and by \cref{upset N_j = F_j}, admissible  filters correspond to closed upsets of the form $\upset N$ for some
  $N \in N(X)$. Therefore, it suffices to show that for each Scott upset $F$, there exists a nuclear subset
  $N \subseteq X$ such that $F = \upset N$. Let $N = \cl (F \cap Y)$. Then $N$ is nuclear by \cref{subsets are nuclear}.
  Moreover, since $F$ is a Scott upset, $\min F \subseteq Y$, and hence
  $F = \upset \min F = \upset \cl (F \cap Y) = \upset N$. 
\end{proof}

\begin{remark}
	Our proof that Scott open filters are admissible relies on \cref{prelim fact-min max exist}, which requires the Axiom of Choice. An alternative proof using only the Prime Ideal Theorem can be found in \cite[Rem.~5.8]{BMRS25}.
\end{remark}

One of the most studied nuclei is the nucleus of double-negation. For a frame $L$ and $a\in L$, recall that the
\emph{pseudocomplement} of $a$ is given by
\[
	a^* = \bigvee\{b \in L \mid b \wedge a = 0\}.
\]
The map $a \mapsto a^{**}$ is the \emph{double-negation nucleus}, and the corresponding sublocale
\[
	\mathfrak B(L) \coloneqq \{a \in L \mid a = a^{**}\}
\]
is the \emph{Booleanization} of $L$ (see, e.g., \cite[p.~246]{PicadoPultr2021}).

If $j$ is the double-negation nucleus, then $j$-dense elements are simply called \emph{dense} (see, e.g.,
\cite[p.~131]{RasiowaSikorski1963}). It is well known that the corresponding admissible filter dually corresponds to
$\max X$, and so we have:
\begin{proposition} \plabel{negneg theorem}
	Let $j$ be the double-negation nucleus on $L$.
	\begin{enumerate}
		\item $H_{j} = \max X$. \label[negneg theorem]{Fnegneg = max X}
		\item $N_j = \max X$. \label[negneg theorem]{Nnegneg=max X}
	\end{enumerate}
\end{proposition}

\begin{proof}
	For \ref{Fnegneg = max X} see, e.g., \cite[Sec.~3]{Guram2001}. For \ref{Nnegneg=max X} observe that
	\ref{Fnegneg = max X} and \cref{upset N_j = F_j} yield $\max X = H_j = \upset N_j$. Consequently, $N_j = \max X$.
\end{proof}

The following definition is well known. For parts \ref{def: dense 1} and \ref{def: dense 2} see, e.g.,
\cite[p.~50]{Johnstone1982} and for part \ref{def: dense 3} see, e.g., \cite[p.~108]{BezhBezh2022}.

\begin{definition}
	Let $L$ be a frame and $X$ its Priestley space.
	\begin{enumerate}
		\item\label{def: dense 1}$j \in N(L)$ is \emph{dense} if $j0 = 0$.
		\item\label{def: dense 2}$S \in S(L)$ is \emph{dense} if $0 \in S$.
		\item\label{def: dense 3}$N \in N(X)$ is \emph{cofinal} if $\max X \subseteq N$.
	\end{enumerate}
\end{definition}

\begin{lemma} \label{lem: dense = cofinal}
	Let $L$ be a frame, $X$ its Priestley space, and $j \in N(L)$. The following are equivalent.
	\begin{enumerate}
		\item $j$ is dense.
		\item $S_j$ is dense.
		\item $N_j$ is cofinal.
	\end{enumerate}
\end{lemma}

\begin{proof}
	The equivalence (1)$\Leftrightarrow$(2) is obvious, and (1)$\Leftrightarrow$(3) is proved in
	\cite[Thms.~23(2) and 28(2)]{BezhGhilardi2007}.
\end{proof}

\begin{theorem} \label{max X least cofinal}
	Let $X$ be an L-space. Then $\max X$ is the least cofinal nuclear subset.
\end{theorem}
\begin{proof}
	By \cref{Nnegneg=max X}, $\max X$ is a nuclear subset of $X$, and clearly it is the least such containing $\max X$.
	Thus, it is the least cofinal nuclear subset of $X$.
\end{proof}

As a consequence, we obtain the following well-known result of Isbell (see, e.g., \cite[p.~40]{PicadoPultr2012}):

\begin{corollary}[Isbell's Density Theorem]
	For a frame $L$, the Booleanization $\mathfrak B(L)$ is the least dense sublocale of $L$.
\end{corollary}

\begin{proof}
	Let $S_j \subseteq L$ be a dense sublocale. Therefore, $N_j$ is cofinal by \cref{lem: dense = cofinal}, and so
	$\max X \subseteq N_j$ by \cref{max X least cofinal}. Consequently, $\mathfrak B(L) \subseteq S_j$ by
	\cref{S(L)=N(X),Nnegneg=max X}.
\end{proof}

\section{Priestley duality for the \texorpdfstring{$d$}{d}-nucleus}\label{section 4}

In this section we describe Priestley duals of inductive nuclei, introduced and studied by Martinez and Zenk
\cite{MartinezZenk2003}. We use the Priestley duality tools from previous sections to study the most prominent inductive
nucleus, known as the $d$-nucleus. Among other things, we characterize the nuclear set $N_d$ corresponding to the
$d$-nucleus, and its localic part $Y_d$.

\begin{definition}[{\cite[Sec.~4]{MartinezZenk2003}}]
	A nucleus $j$ on an algebraic frame $L$ is \emph{inductive} if for all $a \in L$ we have
	\[
		ja = \bigvee \{jk \mid k \in K(L) \mbox{ and } k \leq a\}.
	\]
\end{definition}

Let $X$ be the Priestley space of $L$. As we pointed out after \cref{NL = NX remark}, we identify $N(L)$ with
$N(\clopup(X))$, so for each $j\in N(L)$ there is a unique $N \in N(X)$ such that
$j\,U = X \setminus \downset (N \setminus U)$ for each $U \in \clopup(X)$. 

\begin{definition} \label{def jcore}
	Let $X$ be an L-space. For $j \in N(\clopup(X))$ and $U \in \clopup(X)$, define the \emph{$j$-core of $U$} by
	\[
		{\core_j}U = \bigcup\{j\,V \mid V \in \clopsup(X) \mbox{ and } V \subseteq U\}.
	\]
\end{definition}

\begin{remark}
	Let $j$ be a nucleus on an arithmetic frame $L$ and let $X$ be the Priestley space of $L$. Then $X$ is an arithmetic
	L-space by \cref{thm: another duality}. Therefore, for all clopen upsets $U,V$ of $X$,
	\begin{align*}
		\core_j U \cap \core_j V &= \bigcup\{j\,U' \mid U' \in \clopsup(X),\ U' \subseteq U\} \\
		&\qquad\qquad\cap \bigcup\{j\,V' \mid V' \in \clopsup(X),\ V' \subseteq U\}  \\
		&= \bigcup\{j\,(U' \cap V') \mid U',V' \in \clopsup(X),\ U' \subseteq U,\, V' \subseteq V\} \\
		&=\bigcup\{j\,(W) \mid W \in \clopsup(X),\ W \subseteq U \cap V\} \\
		&=\core_j (U \cap V),
	\end{align*}
	where the second to last equality is a consequence of $X$ being an arithmetic L-space (see \cref{def arithemitc}).
\end{remark}

\begin{definition}
	Let $X$ be an L-space. We call $N \in N(X)$ \emph{inductive} if $\upset(F \cap N)$ is a Scott upset for each Scott
	upset $F$ of $X$.
\end{definition}

As the name suggests, a nuclear subset is inductive iff its corresponding nucleus is inductive. To prove this, we recall
the following:
\begin{lemma}\plabel{lem:key props of Scott upsets}
	Let $X$ be an L-space and $Y$ its localic part.
	\begin{enumerate}
		\item {\upshape({\cite[Lem.~5.1]{BezhanishviliMelzer2022}})}A closed upset $F$ of $X$ is a Scott upset iff
		$F \subseteq \cl U$ implies $F \subseteq U$ for all open upsets $U$ of $X$.
		\label[lem:key props of Scott upsets]{Scott upset lemma}
		\item {\upshape(\cite[Lem.~4.14]{BezhanishviliMelzer2022b})}Let $y \in Y$ and $U$ an open upset of $X$. Then
		$y \in U$ iff $y \in \cl U$.\label[lem:key props of Scott upsets]{y and closures}
	\end{enumerate}
\end{lemma}

\begin{theorem} \plabel{inductive and core}
	Let $X$ be an algebraic L-space and $N \in N(X)$. The following are equivalent.
	\begin{enumerate}
		\item $N$ is inductive. \label[inductive and core]{inductive and core:1}
		\item $j_N\,U = \cl \core_{j_N}U$, for all $U \in \clopup(X)$. \label[inductive and core]{inductive and core:2}
		\item $j_N$ is inductive. \label[inductive and core]{inductive and core:3}
	\end{enumerate}
\end{theorem}

\begin{proof}
	\ref{inductive and core:1}$\Rightarrow$\ref{inductive and core:2}
	Let $U \in \clopup(X)$. Clearly $\cl \core_{j_N}U \subseteq j_N\,U$. For the other inclusion, since $Y$ is dense in
	$X$ (every algebraic L-space is an SL-space), it is sufficient to show that $j_N(U) \cap Y \subseteq \core_{j_N} U$. Suppose
	$y \in j_N(U) \cap Y$. Then $\upset y \cap N \subseteq U$ by \cref{Nj-2}. Since $\upset y$ is a Scott upset and $N$ is
	inductive, $\upset (\upset y \cap N)$ is a Scott upset. But 
	\[
		\upset (\upset y \cap N) \subseteq U = \cl \core U,
	\]
	and
	hence $\upset(\upset y \cap N) \subseteq \core U$ by \cref{Scott upset lemma}. Therefore, since finite unions of
	clopen Scott upsets are clopen Scott upsets, by compactness there exists $V \in \clopsup(X)$ such that
	$\upset y \cap N \subseteq V$ and $V \subseteq U$. Thus, $y \in j_N\,V$ by \cref{Nj-2}, and so $y \in \core_{j_N}U$,
	proving that $j_N(U) \cap Y \subseteq \core_{j_N} U$.
	
	\ref{inductive and core:2}$\Rightarrow$\ref{inductive and core:1}
	Let $F$ be a Scott upset and $\upset(F \cap N) \subseteq \cl U$ for some open upset $U$. Since $X$ is an L-space,
	$U' = \cl U \in \clopup(X)$. Let $y \in \min F$. Then 
	\[
		\upset y \cap N \subseteq \upset (F \cap N) \subseteq U',
	\]
	so
	$y \in j_N\,U'$ by \cref{Nj-2}. Thus, $y \in \cl \core_{j_N}U'$ by \ref{inductive and core:2}, and hence
	${y \in \core_{j_N}U'}$ by \cref{y and closures}. Therefore, there exists $V_y \in \clopsup(X)$ such that $y \in j_N\,V_y$ and
	${V_y \subseteq U'}$. We have $F = \upset \min F \subseteq \bigcup \{ j_N\,V_y \mid y \in \min F \}$, so by compactness
	there are $V_1, \dots, V_n \in \clopsup(X)$ such that $F \subseteq j_N\,V_1 \cup \dots \cup j_N\,V_n$. Let
	$V = V_1 \cup \dots \cup V_n$. Then $V \in \clopsup(X)$. Furthermore, $F \subseteq j_N V$ since $j_N$ is
	order-preserving, and clearly $V \subseteq U'$. The latter together with \cref{Scott upset lemma} yields that
	$V \subseteq U$. Since $F \subseteq j_N\,V$, we have $F \cap N \subseteq j_N(V) \cap N = V \cap N$ by \cref{Nj-1}.
	Consequently, $\upset (F \cap N) \subseteq U$, and hence $\upset (F \cap N)$ is a Scott upset by
	\cref{Scott upset lemma}.

	\ref{inductive and core:2}$\Leftrightarrow$\ref{inductive and core:3}
	Observe that
	\[
		j_N \mbox{ is inductive} \iff j_N\,U = \bigvee \{j_N\,V \mid V \in K(\clopup(X)) \mbox{ and } V \leq U\}.
	\]
	But $K(\clopup(X)) = \clopsup(X)$ (see \cref{rem: 2.8}) and $\bigvee \mathcal U = \cl \bigcup \mathcal U$ for
	$\mathcal U \subseteq \clopup(X)$ (see \cref{prelim fact-join and meet}). Therefore,
	\[
		\bigvee \{j_N\,V \mid V \in K(\clopup(X)) \mbox{ and } V \leq U\} = \cl \core_{j_N} U.
	\]
	Consequently, $j_N$ is inductive iff $U = \cl \core_{j_N}U$ for all $U \in \clopup(X)$.
\end{proof}	

A prominent example of an inductive nucleus is the so-called $d$-nucleus, introduced by Martinez and Zenk
\cite[Sec.~5]{MartinezZenk2003} as a frame-theoretic tool to study $d$-ideals of Riesz spaces (see
\cite{HuijsmansPagter1980} and \cite[Rem.~5.6]{MartinezZenk2003}).
		
\begin{definition} \label{def: d-elements}
	Let $L$ be an algebraic frame.
	\begin{enumerate}
		\item Define $d : L \to L$  by $da = \bigvee \{k^{**} \mid k \in K(L) \mbox{ and } k \leq a\}$ for all $a \in L$.
		\item We call $a \in L$ a \emph{$d$-element} if $da = a$.
	\end{enumerate}
\end{definition}

We write $L_d$ for the fixpoints of $d$.

\begin{theorem}[{\cite[Sec.~5]{MartinezZenk2003}}]\plabel{martinez}
	Let $L$ be an algebraic frame.
	\begin{enumerate}
		\item $d$ is a closure operator on $L$.
	\end{enumerate}
	If in addition $L$ is arithmetic, then
	\begin{enumerate}[resume]
		\item $d$ is an inductive dense nucleus on $L$.\label[martinez]{d inductive dense}
		\item $L_d$ is a sublocale of $L$ and is an arithmetic frame.\label[martinez]{Ld ari}
	\end{enumerate}
\end{theorem}

We next describe the nuclear subset $N_d$ of the Priestley space of $L$ corresponding to the sublocale $L_d$. As in
\cref{remark: Priestley space of sublocale}, for each nucleus $j$ on a frame $L$, we view the corresponding nuclear
subset $N_j$ as the Priestley space of the sublocale $L_j$.

\begin{proposition}\plabel{lem: Nd}
	Let $L$ be a frame, $X$ its Priestley space, and $Y$ the localic part of $X$.
	\begin{enumerate}
		\item\label[lem: Nd]{Y_j = Y cap N_j} If $j \in N(L)$, then $N_j \cap Y$ is the localic part of $N_j$.
	\end{enumerate}
	If in addition $L$ is an arithmetic frame, then
	\begin{enumerate}[resume]
		\item\label[lem: Nd]{Nd cofinal inductive}$N_d$ is a cofinal inductive nuclear subset of $X$.
    \item\label[lem: Nd]{Xd arithmetic}$N_d$ is an arithmetic L-space.
		\item\label[lem: Nd]{Xd = cl rmax Y}$N_d = \cl(N_d \cap Y)$.
	\end{enumerate}
\end{proposition}
\begin{proof}
	\ref{Y_j = Y cap N_j}
	Let $Y_j$ be the localic part of $N_j$. We need to show that $Y_j = N_j \cap Y$. It is straightforward to see that
	$N_j \cap Y \subseteq Y_j$. It remains to show that $Y_j \subseteq Y$. Suppose $y \in Y_j$. Then $\downset y \cap N_j$
	is clopen in $N_j$. Therefore, there is clopen $U \subseteq X$ such that $\downset y \cap N_j = U \cap N_j$. Since
	$N_j$ is a nuclear set, $\downset (U \cap N_j)$ is clopen in $X$. But
	$\downset (U \cap N_j) = \downset (\downset y \cap N_j) = \downset y$ because $y \in N_j$. Thus, $y \in Y$.

  \ref{Nd cofinal inductive}
  Apply \cref{d inductive dense,inductive and core,lem: dense = cofinal}.

	\ref{Xd arithmetic}
	Since $L$ is arithmetic, so is $L_d$ by \cref{Ld ari}. Therefore, the result follows from \cref{thm: another duality}
	because $N_d$ is the Priestley dual of $L_d$.
	
	\ref{Xd = cl rmax Y}
	$N_d \cap Y$ is the localic part of $N_d$ by \ref{Y_j = Y cap N_j}, and $N_d$ is an SL-space by \ref{Xd arithmetic}.
	Hence, $\cl(N_d \cap Y) = N_d$.
\end{proof}

Let $X$ be an arithmetic L-space. Since $\clopsup(X) = K(\clopup(X))$ (see \cref{rem: 2.8}),
$d : \clopup(X) \to \clopup(X)$ is given by
\[
	d\,U = \cl\bigcup \{ V^{**} \mid V \in \clopsup(X) \mbox{ and } V \subseteq U \},
\]
where $V^* = X \setminus \downset V$ (see, e.g., \cite[p.~20]{Esakia2019}), so $x \in V^{**}$ iff
$\upset x \subseteq \downset V$. We also recall (see \cref{def jcore}) that the $d$-core of $U \in \clopup(X)$ is given
by 
\[
	\dcore U = \bigcup \{d\,V \mid V \in \clopsup(X),\, V \subseteq U\}.
\]

\begin{lemma} \plabel{lem: maxY}
	Let $L$ be an arithmetic frame, $X$ its Priestley space, and $Y$ the localic part of $X$.
	\begin{enumerate}
		\item\label[lem: maxY]{dU-1} If $U \in \clopsup(X)$, then $d\,U = U^{**}$.
		\item\label[lem: maxY]{dcore condition} If $U \in \clopup(X)$, then $x \in \dcore U$ iff
		$\upset x \subseteq \downset \core U$.
	\end{enumerate}
\end{lemma}
\begin{proof}
	\ref{dU-1}
	$d\,U = \cl\bigcup \{ V^{**} \mid V \in \clopsup(X) \mbox{ and } V \subseteq U \} = U^{**}$ since $U \in \clopsup(X)$.

	\ref{dcore condition}
	First suppose that $x \in \dcore U$. Then there is $V \in \clopsup(X)$ with $x \in d\,V$ and $V \subseteq U$.
	Therefore, $x \in V^{**}$ by \ref{dU-1}, which means that $\upset x \subseteq \downset V$. Since $V$ is a Scott upset,
	$V \subseteq U$ implies $V \subseteq \core U$. Thus, $x \in d\,V$ implies $\upset x \subseteq \downset \core U$.
	For the converse, if $\upset x \subseteq \downset \core U$, then
	\[
		\upset x \subseteq {\big\downarrow} \bigcup \{V \in \clopsup(X) \mid V \subseteq U\}
		= \bigcup \{\downset V \mid V \subseteq U \mbox{ and } V \in \clopsup(X)\}.
	\]
	Hence, by \cref{prelim fact-clopen}, $\{\downset V \mid V \subseteq U \mbox{ and } V \in \clopsup(X)\}$ is an open
	cover of $\upset x$. Since $\upset x$ is compact and this open cover is directed, there is $V \in \clopsup(X)$ such
	that $V \subseteq U$ and $\upset x \subseteq \downset V$. This yields that $x\in V^{**}$, so $x\in d\,V$ by
	\ref{dU-1}. Consequently, $x \in \dcore U$.
\end{proof}

Let $X$ be an arithmetic L-space and $Y$ its localic part. We let $Y_d$ denote the localic part of $N_d$. By
\cref{Y_j = Y cap N_j},  $Y_d = N_d \cap Y$. We conclude the section by giving several characterizations of $Y_d$. For
this we need the following lemma. 

\begin{lemma}\label{scott intersection of clopen scott}
	Let $X$ be an algebraic L-space and $F$ a Scott upset of $X$. Then
	\[
		F = \bigcap \{U \in \clopsup(X) \mid F \subseteq U\}.
	\]
\end{lemma}

\begin{proof}
	Since $F$ is a closed upset, $F = \bigcap \{U \in \clopup(X) \mid F \subseteq U\}$ (see
	\cref{prelim fact-intersections}). Thus, it suffices to show that for each $U \in \clopup(X)$ with $F \subseteq U$,
	there is ${V \in \clopsup(X)}$ with $F \subseteq V \subseteq U$. Since $X$ is an algebraic L-space, $U = \cl \core U$,
	so $F \subseteq U$ implies that $F \subseteq \core U$ by \cref{Scott upset lemma}. Now apply compactness to obtain the
	desired~$V$.
\end{proof}

\begin{theorem} \plabel{eqv conditions rmax}
	Let $L$ be an arithmetic frame, $X$ its Priestley space, and $Y$ the localic part of $X$. For $y \in Y$, the following
	are equivalent.
	\begin{enumerate}
		\item\label[eqv conditions rmax]{eqv rmax 1} $y \in Y_d$.
		\item\label[eqv conditions rmax]{eqv rmax 2} $\forall U \in \clopup(X) , \, y \in \dcore U \implies y \in U$.
		\item\label[eqv conditions rmax]{eqv rmax 3}
		$\forall V \in \clopsup(X), \, \max \upset y \subseteq V \implies y \in V$.
		\item\label[eqv conditions rmax]{eqv rmax 4} $\{y\} = \max(\downset x \cap Y)$ for some $x \in \max X$.
	\end{enumerate}
\end{theorem}

\begin{proof}
	\ref{eqv rmax 1}$\Rightarrow$\ref{eqv rmax 2}
	If $y \in \dcore U$, then $y \in d\,U$. Therefore, since $y\in Y_d \subseteq N_d$, \cref{Nj-1} implies that $y\in U$.

	\ref{eqv rmax 2}$\Rightarrow$\ref{eqv rmax 3}
	Suppose $\max \upset y \subseteq V$. Then $\upset y \subseteq \downset V$. Since $V$ is a clopen Scott upset,
	$V = \core V$ by \cref{rem: 2.8}. Therefore, $\downset V = \downset \core V$, and hence $y \in \dcore V$ by
	\cref{dcore condition}. Thus, $y \in V$ by \ref{eqv rmax 2}.

	\ref{eqv rmax 3}$\Rightarrow$\ref{eqv rmax 4}
	Suppose that for every $x \in \max \upset y$ there is $y' \in \downset x \cap Y$ with $y < y' \leq x$. Then
	$y' \nleq y$, so \cref{scott intersection of clopen scott} implies that there is $V_x \in \clopsup(X)$ with
	$y' \in V_x$ and $y \not \in V_x$. Therefore, $\max \upset y \subseteq \bigcup V_x$, and since $\max \upset y$ is
	closed (see \cref{prelim fact-max closed}) and the open cover is directed, there is $V \in \clopsup(X)$ containing
	$\max \upset y$ and missing $y$, a contradiction.

	\ref{eqv rmax 4}$\Rightarrow$\ref{eqv rmax 1}
	It is sufficient to show that $da \in y$ implies $a \in y$ for each $a \in L$, and hence it is enough to show that
	$y \in d\,U$ implies $y \in U$ for each $U \in \clopup(X)$. Let $y \in d\,U$. Then $y \in \cl \dcore U$ since $d$ is
	inductive. Therefore, $y \in \dcore U$ by \cref{y and closures}. Thus, by \cref{dU-1}, $y \in d\,V = V^{**}$ for some
	$V \in \clopsup(X)$ with $V \subseteq U$. Hence, $\upset y \subseteq \downset V$. By \ref{eqv rmax 4}, there is
	$x \in \max X$ with $\{y\} = \max(\downset x \cap Y)$. But then $x \in V$, and since $V$ is a Scott upset, there is
	$y' \in V \cap Y$ with $y' \leq x$. Consequently, $y' \leq y$, and so $y \in V \subseteq U$.
\end{proof}

\section{\texorpdfstring{$\max Y$}{max Y} and regularity of \texorpdfstring{$L_d$}{Ld}}
\label{new section 5}

Martinez and Zenk \cite[Prop.~5.2]{MartinezZenk2003} characterized when $L_d$ is a regular frame. In this section we
give several alternative characterizations, utilizing Priestley duality. This, in particular, involves the maximal
spectrum $\max Y$ of the localic part $Y$ of the Priestley space of $L$. As a consequence, we obtain that $L_d$ is
regular iff $L_d$ is locally Stone. 

Recall (see, e.g., \cite[p.~89]{PicadoPultr2012}) that a frame $L$ is \emph{regular}  if for all $a \in L$ we have
\[
	a = \bigvee \{b \in L \mid b^* \vee a = 1\}.
\]
 
Priestley spaces of regular frames were studied in
\cite{PultrSichler1988,BezhanishviliGabelaiaJibladze2016,BezhanishviliMelzer2022b}. We recall:
\begin{definition}[{\cite[Defs.~7.1 and 7.6]{BezhanishviliMelzer2022b}}]
	Let $X$ be an L-space.
	\begin{enumerate}
		\item For $U \in \clopup(X)$, the \emph{regular part} of $U$ is
		\[
			\reg U = \bigcup \{V \in \clopup(X) \mid \downset V \subseteq U\}.
		\]
		\item $X$ is \emph{L-regular} if $\cl \reg U = U$ for each $U \in \clopup(X)$.
	\end{enumerate}
\end{definition}

\begin{theorem} \plabel{Lregular facts}
	Let $L$ be a frame, $X$ its Priestley space, and $Y$ the localic part of $X$.
	\begin{enumerate}
		\item \textup{\cite[Lem~3.6]{BezhanishviliGabelaiaJibladze2016}} $L$ is regular iff $X$ is L-regular.
		\label[Lregular facts]{regular iff}
		\item \textup{\cite[Proof of Thm.~7.11]{BezhanishviliMelzer2022b}} If $X$ is an SL-space and $Y$ is regular, then
		$X$ is L-regular.\label[Lregular facts]{Y regular implies X regular}
		\item \textup{\cite[Lem.~7.15(3)]{BezhanishviliMelzer2022b}} If $X$ is L-regular, then $Y \subseteq \min X$.
		\label[Lregular facts]{reg implies Y subset min}
	\end{enumerate}	
\end{theorem}

An element $p \neq 1$ of a frame $L$ is \emph{\textup{(}meet-\textup{)}prime} if $a \wedge b \leq p$ implies $a \leq p$
or $b \leq p$ (see, e.g., \cite[p.~13]{PicadoPultr2012}). A prime element $p$ is \emph{minimal prime with respect to}
$a \in L$ if $p$ is minimal among the primes $q \geq a$. It is known (see, e.g., \cite[p.~264]{NiefeldRosenthal1987})
that every prime element $p$ greater than $a \in L$ has a minimal prime element $q$ with respect to $a$ beneath it.
Since the assignment $p \mapsto L \setminus \downset p$ establishes an isomorphism between the posets of prime elements
and completely prime filters  (see, e.g., \cite[p.~14]{PicadoPultr2012}), this condition can equivalently be formulated
as follows: for every completely prime filter $P$ contained in a filter $F$, there exists a completely prime filter $Q$
that is maximal among the completely prime filters contained in $F$. Thus, we arrive at the following lemma, which gives
the means to find (relatively) maximal localic points.

\begin{lemma}\plabel{lem: ext}
	Let $L$ be a frame, $X$ its Priestley space, $Y$ the localic part of $X$, and $y \in Y$.
	\begin{enumerate}
		\item $\upset y \cap \max Y \neq \varnothing$.\label[lem: ext]{extend y to max}
		\item $\upset y \cap \max (\downset x \cap Y) \neq \varnothing$ for every $x \in X$ with $y \leq x$.
		\label[lem: ext]{extend y to max relative}
	\end{enumerate}
\end{lemma}
\begin{proof}
	For \ref{extend y to max} take $P = y$ and $F = L$, and for \ref{extend y to max relative} take $P = y$ and $F = x$.
\end{proof}

We show that $\max Y\subseteq Y_d$, but that the converse is not true in general. For this we require the following lemma.

\begin{lemma} \plabel{lem: scottupsets}
	Let $L$ be an arithmetic frame, $X$ its Priestley space, and $Y$ the localic part of $X$.
	\begin{enumerate}
						\item\label[lem: scottupsets]{intersection of Scott upsets} If $F$ and $G$ are Scott upsets of $X$, then so is $F \cap G$.
		\item\label[lem: scottupsets]{max y go in scott upsets} Suppose $y \in \max Y$ and $F$ is a Scott upset of $X$. If $\upset y \cap F \neq \varnothing$, then $y \in F$.
	\end{enumerate}
\end{lemma}
\begin{proof}
	\ref{intersection of Scott upsets} This can be seen by applying \cite[Lem.~5.2]{BezhanishviliMelzer2023} and \cite[Lem.~6.3(2)]{BezhanishviliMelzer2022b}. To keep the proof self-contained, we give a short argument. By \cref{scott intersection of clopen scott}, $F$ and $G$ are intersections of down-directed families of clopen Scott upsets. Therefore, so is $F \cap G$ since the binary intersection of clopen Scott upsets is a clopen Scott upset ($X$ is an arithmetic L-space). Thus, $F \cap G$ is a Scott upset because the intersection of a down-directed family of clopen Scott upsets is a Scott upset (see \cite[Lem.~5.14(1)]{BezhanishviliMelzer2022b}).
    
	\ref{max y go in scott upsets}
		Since $\upset y \cap F$ is a nonempty closed upset, $\min(\upset y \cap F) \neq \varnothing$ (see, e.g., \cite[Thm.~3.2.1]{Esakia2019}). Because $\upset y,F$ are Scott upsets, $\upset y \cap F$ is also a Scott upset by \ref{intersection of Scott upsets}. Therefore, $\min(\upset y \cap F) \subseteq Y$. Since $y$ is underneath each point in $\min(\upset y \cap F)$ and
		$y \in \max Y$, we conclude that $\min(\upset y \cap F) = \{ y \}$, and hence $y \in F$.
				\end{proof}

\begin{theorem} \plabel{max Y subset cofinal inductive}
	Let $X$ be an algebraic L-space and $Y$ its localic part. \begin{enumerate}
	\item If $N \subseteq X$ is a cofinal inductive nuclear subset, then $\max Y \subseteq N$. \label{Y subset N}
	\item If $X$ is the Priestley space of an arithmetic frame $L$, then $\max Y \subseteq Y_d$. \label[max Y subset cofinal inductive]{Y subset Y_d}
\end{enumerate}
\end{theorem}

\begin{proof}
	\ref{Y subset N} Let $y \in \max Y$. Since $\upset y$ is a Scott upset and $N$ is inductive, $\upset (\upset y \cap N)$ is a Scott upset. Because $N$ is cofinal, $\max X \subseteq N$, and thus $\upset y \cap N \neq \varnothing$.
	Therefore, $\upset (\upset y \cap N) \subseteq \upset y$ is a nonempty Scott upset, so $y \in \upset (\upset y \cap N)$ by \cref{max y go in scott upsets}. Hence, $\upset y \subseteq \upset (\upset y \cap N)$.
	Consequently, $\upset y = \upset (\upset y \cap N)$, and so $y \in N$.
		
	\ref{Y subset Y_d} By \cref{Nd cofinal inductive}, $N_d$ is a cofinal inductive nuclear subset of $X$. Therefore, $\max Y \subseteq N_d$ by \ref{Y subset N}, and so $\max Y \subseteq Y_d$ by \cref{Y_j = Y cap N_j}.
\end{proof}

\begin{example} \label{ex: 4.11}
	To see that in general $\max Y\ne Y_d$, let $\beta \mathbb N$
		be the Stone-Čech compactification of the natural numbers  (see, e.g., \cite[p.~174]{Engelking1989}). As is customary, we write $\mathbb N^*$ for the remainder. Let $X = \beta \mathbb N \cup \{y\}$, where the order on $X$ is defined as shown in \cref{fig 1:a}. It is well known (see, e.g., \cite[p.~28]{Dwinger1961}) that $\beta \mathbb N$ is homeomorphic to the Stone space of the powerset $\wp(\mathbb N)$ of $\mathbb N$.
Therefore, $X$ is homeomorphic to the Priestley space of the lattice $L$ obtained by adding a new top to $\wp(\mathbb N)$; see \cref{fig 1:b}. Since $\wp(\mathbb N)$ is an arithmetic frame and $1 \in K(L)$, it is clear that so is $L$, and hence $X$ is an arithmetic L-space.

Because the set of isolated points of $X$ is $\mathbb N\cup\{y\}$, we have that $\downset x$ is clopen iff $x$ is an isolated point of $X$. Thus, the localic part of $X$ is $Y= \mathbb{N} \cup \{y\}$. Therefore, $\max Y = \mathbb{N}$. 
On the other hand, $y \in N_d$ by \cref{eqv rmax 4}, so $Y_d = Y$. 
Consequently, $Y_d \ne \max Y$.
	\begin{figure}[H]
	\begin{subfigure}{.48\textwidth}
	\centering
	\begin{tikzpicture}[
	bullet/.style={fill,circle,inner sep=2pt}]
	\begin{scope}[nodes=bullet]
		\node[label=above:0] (0) at (0,0) {};
		\node[label=above:1] (1) at (1,0) {};
		\node[label=above:2] (2) at (2,0) {};
		\node[label=below:$y$] (y) at (3,-2) {};
	\end{scope}
		\path (2) -- node[auto=false]{\ldots} (4,0);
		\draw[very thick, |-|] (4, 0) to node[label=above:$\mathbb N^*$] {} (7,0);
		\draw (y) to (0);
		\draw (y) to (1);
		\draw (y) to (2);
		\draw (y) to (4,0);
		\draw (y) to (7,0);
	\end{tikzpicture}
	\caption{The Priestley space $\beta \mathbb N \cup \{y\}.$} \label{fig 1:a}
	\end{subfigure}
	\begin{subfigure}{.48\textwidth}
	\centering
	\begin{tikzpicture}[
	bullet/.style={fill,circle,inner sep=2pt}]
		\draw[thick] (0,0) ellipse (1 and 1.25) node at (0,0) {$\wp(\mathbb N)$};
		\begin{scope}[nodes=bullet]
		\node (0) at (0,-1.25) {};
		\node (1) at (0,1.25) {};
		\node (2) at (0,1.75) {};
		\draw (1) to (2);
		\end{scope}
	\end{tikzpicture}
	\subcaption{The lattice $\wp(N)$ with a new top.} \label{fig 1:b}
	\end{subfigure}
	\caption{An arithmetic L-space in which $Y_d \neq \max Y$.} \label{fig 1}
	\end{figure}
\end{example}

In order to characterize when $L_d$ is regular, we require the following lemma, which allows us  to separate Scott upsets from maximal localic points via clopen Scott upsets.

\begin{lemma}\label{y not in U gives disjoint V}
	Let $X$ be an arithmetic L-space, $Y$ its localic part, and $F$ a Scott upset of $X$.
		If $y \in \max Y \setminus F$, then there is $U \in \clopsup(X)$ with $y \in U$ and $U \cap F = \varnothing$.
\end{lemma}

\begin{proof}
Since $y \not \in F$, we have $\upset y \cap F = \varnothing$ by \cref{max y go in scott upsets}.
Therefore, \cref{scott intersection of clopen scott} yields that $\bigcap \{U \in \clopsup(X) \mid y \in U\} \cap F = \varnothing$.
Thus, we can use compactness and the fact that finite intersections of clopen Scott upsets are again Scott upsets ($X$ is an arithmetic L-space) to produce $U \in \clopsup(X)$ with $y \in U$ and $U \cap F = \varnothing$.
\end{proof}

We recall that a space $X$ is \emph{locally Stone} if $X$ is zero-dimensional, locally compact, and Hausdorff. Thus, $X$ is locally Stone if in the definition of a Stone space we weaken compactness to local compactness. A frame $L$ is \emph{locally Stone} if it is isomorphic to the frame of opens of a locally Stone space. By \cite[Thm.~3.11]{BezhKornell2023}, a frame is locally Stone iff it is algebraic and zero-dimensional (each element is a join of complemented elements).
We will use the following fact: if $L$ is an algebraic frame and $X$ its Priestley space, then the localic part $Y$ of $X$ is locally compact (see \cite[Thm.~5.10]{BezhanishviliMelzer2022b}, where the result is proved in the more general setting of continuous frames). 

\begin{theorem} \label{thm: Ld regular}
 	Let $L$ be an arithmetic frame, $X$ its Priestley space, and $Y$ the localic part of $X$. The following are equivalent.
 	\begin{enumerate}
 		\item $L_d$ is regular.
 		\item $N_d$ is L-regular.
 		\item $Y_d$ is an antichain.
 		\item $\max Y = Y_d$.
 		\item $Y_d$ is a locally Stone space.
        \item $L_d$ is a locally Stone frame.
 	\end{enumerate}
\end{theorem}

\begin{proof}
	(1)$\Leftrightarrow$(2)
	This is immediate from \cref{regular iff} since $N_d$ is the Priestley space of~$L_d$.

	(2)$\Rightarrow$(3) If $N_d$ is L-regular, then $Y_d \subseteq \min N_d$ by \cref{reg implies Y subset min}. Therefore, $Y_d$ is an antichain. 

	(3)$\Rightarrow$(4) By \cref{Y subset Y_d}, $\max Y \subseteq Y_d$.
		For the converse,
		suppose $y \in Y_d$. By \cref{extend y to max}, there is $y' \in \max Y \cap \upset y$. Then $y' \in \max Y \subseteq Y_d$, so $y = y'$ since $Y_d$ is an antichain by (3). Thus, $y \in \max Y$.

	(4)$\Rightarrow$(5)
	Since $L$ is arithmetic, $L_d$ is arithmetic by \cref{Ld ari}.
	Therefore, as we pointed out above, $Y_d$ is locally compact.
	Recall (see \cref{topology Y}) that open subsets of $Y_d$ are exactly the sets of the form $U \cap Y_d$ for $U \in \clopup(X)$. Hence, (4) implies that open subsets are the sets of the form $U \cap \max Y$. By \cref{y and closures}, \[U \cap \max Y = \cl( \core U) \cap \max Y = \core U \cap \max Y.\]
						Thus, to see that $Y_d$ is zero-dimensional, it is enough to show that $U \cap \max Y$ is clopen for each $U \in \clopsup(X)$. For this it is sufficient to show that for each $y \in \max Y \setminus U$ there is $V \in \clopsup(X)$ with $y \in V \cap \max Y \subseteq \max Y \setminus U$. But this follows from \cref{y not in U gives disjoint V}.
		Finally, to see that $Y_d$ is Hausdorff, let $y, y' \in Y_d = \max Y$ be distinct. Then $y \not \in \upset y'$, so by \cref{y not in U gives disjoint V} there is $U \in \clopsup(X)$ such that $y \in U$ and $y' \not \in U$. But then $y \in U \cap Y_d$, which is clopen by the above.

	 (5)$\Leftrightarrow$(6) Because $L_d$ is spatial and $Y_d$ is the space of points of $L_d$ (see \cref{topology Y}), $L_d$ is isomorphic the frame $\Omega(Y_d)$ of open subsets of $Y_d$. Therefore, $Y_d$ is locally Stone iff $L_d$ is locally Stone by \cite[Thm.~3.11]{BezhKornell2023}.

  (5)$\Rightarrow$(2) Since $Y_d$ is locally Stone, $Y_d$ is regular. Hence, $N_d$ is L-regular by \cref{Y regular implies X regular}.
\end{proof}

\section{\texorpdfstring
{Spectra of maximal $d$-elements}
{Spectra of maximal d-elements}} \label{section 5}

In this section, we begin our investigation of the spectrum $\max L_d$ of maximal $d$-elements of an arithmetic frame $L$, as introduced in \cite{Bhattacharjee2018}.
First, we show that $\max L_d$ is in a bijective correspondence with $\min Y_d$. Following this, we establish that $\min Y_d$, viewed as a subspace of $Y$, is homeomorphic to $\max L_d$. The homeomorphism enables us to analyze the properties of $\max L_d$ through $\min Y_d$. We show that the frame $\Omega(\min Y_d)$ of open subsets of $\min Y_d$ can be realized as a sublocale of $L$, and describe the corresponding nuclear subset of $X$. We conclude the section by observing that the localic part of this nuclear subset is the soberification of $\min Y_d$.

\begin{definition} \label{def: d-elements 2}
	Let $X$ be an arithmetic L-space and $U \in \clopup(X)$.
	\begin{enumerate}
				\item We call $U$  a \emph{$d$-upset} if $\cl \dcore U = U$.
		\item We call $U$  a \emph{maximal $d$-upset} if it is maximal among proper $d$-upsets of $X$. 	\end{enumerate}
\end{definition}

\begin{remark}\label{rem: d upsets}
Since $d$ is inductive, it is straightforward to verify that (maximal) $d$-elements of an arithmetic frame correspond
to (maximal) $d$-upsets of its Priestley space.
\end{remark}

We now show that maximal $d$-upsets are in one-to-one correspondence with elements of $\min Y_d$. For this, we require the following lemmas.

\begin{lemma} \label{dU=X condition}
	Let $X$ be an arithmetic L-space and $U \in \clopup(X)$. Then $\cl \dcore U = X$ iff $Y_d \subseteq U$.
\end{lemma}
\begin{proof}
	Since $d$ is inductive, $\cl \dcore U = X$ iff $d\,U = X$ (see \cref{inductive and core:2}). It follows from \cref{j=1 iff contains N_j} that $d\,U = X$ iff $N_d \subseteq U$. But by \cref{Xd = cl rmax Y}, $N_d = \cl Y_d$, so $N_d \subseteq U$ iff $Y_d \subseteq U$ since $U$ is closed.
\end{proof}

\begin{lemma} \plabel{min Yd and d-elements}
	Let $X$ be an arithmetic L-space and $y \in Y_d$.
	\begin{enumerate}
		\item $X \setminus \downset y$ is a $d$-upset. \label[min Yd and d-elements]{Yd is d-element}
		\item $X \setminus \downset y$ is a maximal $d$-upset iff $y \in \min Y_d$.  \label[min Yd and d-elements]{min rmax gives maximal d-elements}
		\item Maximal $d$-upsets are precisely the clopen upsets of the form $X \setminus \downset y$ for some ${y \in \min Y_d}$.\label[min Yd and d-elements]{min Yd is max d-elements}
	\end{enumerate}
\end{lemma}
\begin{proof}
	\ref{Yd is d-element} We need to show $X \setminus \downset y = \cl \dcore (X \setminus \downset y)$. For this it is sufficient to show that ${Y \cap (X \setminus \downset y) = Y \cap \dcore(X \setminus \downset y)}$ since $X$ is a spatial L-space. We have the inclusion ${Y \cap (X \setminus \downset y) \subseteq Y \cap \dcore(X \setminus \downset y)}$ since $Y \cap U = Y \cap \core U$ and $\core U \subseteq \dcore U$ for each clopen upset $U$. For the reverse inclusion, let $z \in Y$ and suppose towards a contradiction that $z \in \dcore (X \setminus \downset y)$ and $z \not \in X \setminus \downset y$. Then $z \leq y$, so $y \in \dcore (X \setminus \downset y)$. Therefore, $y \in X \setminus \downset y$ by \cref{eqv conditions rmax} ($y \in Y_d$), a contradiction.

	\ref{min rmax gives maximal d-elements}  Suppose $y \not \in \min Y_d$. Then there exists $y' \in Y_d$ with $y' < y$. Therefore, there exists $V \in \clopup(X)$ containing $y$ and missing $y'$. Let $U = \cl \dcore (V \cup (X \setminus \downset y))$. Then $U$ is a $d$-upset and $X \setminus \downset y \subsetneq U$ because $y\in U$ and $y \notin X \setminus \downset y$. But $U = \cl \dcore (V \cup (X \setminus \downset y)) \neq X$ by \cref{dU=X condition} since $y' \in Y_d$ and $y' \notin V \cup (X \setminus \downset y)$, yielding that $Y_d \not \subseteq V \cup (X \setminus \downset y)$. Thus, $X \setminus \downset y$ is not a maximal $d$-upset.

    Suppose $y \in \min Y_d$ and $X \setminus \downset y \subsetneq U$ for a $d$-upset $U$. Then $y \in U$, and since $y \in \min Y_d$ we get $Y_d \subseteq (X \setminus \downset y) \cup \{y\} \subseteq U$. Therefore, $U = \cl \dcore U = X$ by \cref{dU=X condition}.

	\ref{min Yd is max d-elements} By \ref{min rmax gives maximal d-elements} it suffices to show that every maximal $d$-upset is of the desired form, so suppose $U \in \clopup(X)$ is a maximal $d$-upset. Then $\cl \dcore U = U \neq X$, so $Y_d \not \subseteq U$ by \cref{dU=X condition}. Therefore, there exists $y \in Y_d \setminus U$. Then $\downset y \cap U = \varnothing$, so $U \subseteq X \setminus \downset y \neq X$. But $X \setminus \downset y$ is a $d$-upset by \ref{Yd is d-element}. Hence, $U = X \setminus \downset y$ since $U$ is a maximal $d$-upset. Thus, $y \in \min Y_d$ by \ref{min rmax gives maximal d-elements}.
\end{proof}

As an immediate consequence, we obtain:
\begin{theorem} \label{thm: minYd = maxLd}
	Let $X$ be an arithmetic L-space. The map $y \mapsto X \setminus \downset y$ is a bijection from $\min Y_d$ to the collection of maximal $d$-upsets of $X$.
\end{theorem}

Equipping $Y_d$ with the subspace topology inherited from $Y$, we have:

\begin{theorem} \label{thm: min Yd = max Ld}
	$\min Y_d$ is homeomorphic to $\max L_d$.
\end{theorem}
\begin{proof}
	Since $\varphi : L \to \clopup(X)$ is an isomorphism, define $\alpha : \min Y_d \to \max L_d$ by $\alpha(y) = \varphi^{-1}(X \setminus \downset y)$. By \cref{thm: minYd = maxLd}, $\alpha$ is a bijection. Thus, it suffices to show that for all $U \subseteq \min Y_d$ we have $U$ is open iff $\alpha(U)$ is open. Now, $U$ is open iff $U = V \cap \min Y_d $ for some $V \in \clopup(X)$, and
 $\alpha(U)$ is open iff $\alpha(U) = \{m \in \max L_d \mid a \not \leq m\}$ for some $a \in L$ (see, e.g., \cite[Sec.~3]{Bhattacharjee2018}). Since $m \in \max(L_d)$ iff
  $\varphi(m)$ is a maximal $d$-upset, by \cref{min Yd is max d-elements} we have that $m \in \max(L_d)$ iff $\varphi(m) = X \setminus \downset y$ for some $y \in \min Y_d$. Moreover, for $a \in L$, we have that $\varphi(a) \not \subseteq X \setminus \downset y$ iff $y \in \varphi(a)$. Therefore,
 \begin{align*}
     \alpha(U) = \{m \in \max L_d \mid a \not\leq m\}
     &\iff \varphi[\alpha(U)] = \{\varphi(m) \mid \varphi(a) \not\subseteq \varphi(m)\}\\
     &\iff \varphi[\alpha(U)] = \{X \setminus \downset y \mid y \in \min Y_d,\ \varphi(a) \not \subseteq  X \setminus \downset y\}\\
     &\iff U = \{y \in \min Y_d \mid y \in \varphi(a)\} \\
     &\iff U = \varphi(a) \cap \min Y_d. \qedhere
 \end{align*}
\end{proof}

As we pointed out in the introduction, if $L$ is an arithmetic frame with a unit (see the beginning of \cref{sec: compactness}), then $\max L_d$ is a compact $T_1$-space. We next show that being $T_1$ does not depend on the existence of a unit.

\begin{proposition} \label{fact: t1}
    $\min Y_d$ is $T_1$.
\end{proposition}
\begin{proof}
    Suppose $y,y' \in \min Y_d$ are distinct. Then $y \not\leq y'$ since $\min Y_d$ is an antichain. By Priestley separation, there exists $U \in \clopup(X)$ with $y \in U$ and $y' \notin U$. Hence, $U \cap \min Y_d$ is an open subset of $\min Y_d$ containing $y$ and missing $y'$.
\end{proof}

We now concentrate on the frame $\Omega(\min Y_d)$ of open subsets of $\min Y_d$ and show that it can be realized as a sublocale of $L$. To this end, since $L$ is isomorphic to $\clopup(X)$, we introduce a nucleus on $\clopup(X)$ that determines $\Omega(\min Y_d)$. 

\begin{lemma} \label{onto frame hom}
    The map $h : \clopup(X) \to \Omega(\min Y_d)$,
    given by $h(U) = U \cap \min Y_d$, is an onto frame homomorphism.
\end{lemma}

\begin{proof}
    It is clear that $h$ is onto and preserves finite meets. 
        To see that it preserves arbitrary joins let $\{U_i\} \subseteq \clopup(X)$. Then, by \cref{y and closures},
    \[
        h\left(\bigvee U_i\right) = \left(\cl \bigcup U_i\right) \cap \min Y_d = \left(\bigcup U_i\right) \cap \min Y_d = \bigcup (U_i \cap \min Y_d) = \bigcup h(U_i). \qedhere
    \]
    \end{proof}

By the previous lemma, there is a nucleus $\ds = h_* \circ h : \clopup(X) \to \clopup(X)$, where $h_*$ is the right adjoint of $h$ (see, e.g., \cite[p.~31]{PicadoPultr2012}). Then, for each $U \in \clopup(X)$, 
\begin{align*}
    \ds (U) 
    &= \bigvee \{V \in \clopup(X) \mid h(V) \subseteq h(U)\} \\
    &= \cl\bigcup \{V \in \clopup(X) \mid V \cap \min Y_d \subseteq U \cap \min Y_d\}\\
    &= \cl\bigcup \{V \in \clopup(X) \mid V \cap \min Y_d \subseteq U\}.
\end{align*}

\begin{lemma} \label{rho on L}
    Let $L$ be an arithmetic frame and $X$ its L-space. For $a \in L$ set $$M_a = \bigwedge \{m \in \max L_d \mid a \leq m\}.$$ Then $\varphi\left( \bigwedge M_a\right ) 
        = \ds(\varphi(a))$.
\end{lemma}
\begin{proof}
    Observe that
    \begin{align*}
        \varphi\left(\bigwedge M_a\right) 
        &= \varphi\left (\bigvee \{b \in L \mid b \leq m \mbox{ for all } m \in M_a\}\right)\\
        &= \cl \bigcup \left\{ V \in \clopup(X) \mid V \subseteq \bigcap \varphi[M_a]\right\}.
    \end{align*}
    Recall that $m \in \max L_d$ iff $\varphi(m)$ is a maximal $d$-upset. Thus, using \cref{min Yd is max d-elements}, we obtain that $m \in M_a$ iff $\varphi(m) = X \setminus \downset y$ for some $y \in \min Y_d \setminus \varphi(a)$. Therefore,
    \[
        \bigcap \varphi[M_a] = \bigcap \{X \setminus \downset y \mid y \in \min Y_d \setminus \varphi(a)\} = X \setminus \downset (\min Y_d \setminus \varphi(a)).
    \]
    Consequently, since $V$ is an upset,
    \begin{align*}
        V \subseteq \bigcap \varphi[M_a]
        &\iff V \subseteq X \setminus \downset (\min Y_d \setminus \varphi(a))\\
        &\iff V \cap \downset(\min Y_d \setminus \varphi(a)) = \varnothing\\
        &\iff V \cap (\min Y_d \setminus \varphi(a)) = \varnothing\\
        &\iff V \cap \min Y_d \subseteq \varphi(a),
    \end{align*}
        and the result follows from the above description of $\rho$.
\end{proof}

\begin{remark}
    By the previous lemma, the nucleus $\rho$ can be defined on an arbitrary arithmetic frame $L$ by \[\ds(a) = \bigwedge \{m \in \max L_d \mid a\leq m\}\] for each $a \in L$.
\end{remark}

We now describe the nuclear subset of $X$ corresponding to the nucleus $\ds$. For this we use the following: 

\begin{theorem} \label{thm: nuclear subset of ds}
    Let $L$ be an arithmetic $L$-space. Then $N_{\ds} = \cl \min Y_d$.
\end{theorem}
\begin{proof}
    By \cref{subsets are nuclear}, $\cl \min Y_d$ is a nuclear subset of $X$. Let $j \in N(L)$ be the nucleus associated with $\cl \min Y_d$ (see \cref{NL = NX remark}). It suffices to show that $\varphi(j(a)) = \ds(\varphi(a))$ for all $a \in L$. 

    ($\subseteq$) Let $x \in \varphi(j(a))$. Then $\upset x \cap \cl \min Y_d \subseteq \varphi(a)$ by \cref{Nj-2}. 
    Since $\upset x$ is a closed upset, it is an intersection of clopen upsets (see \cref{prelim fact-intersections}). 
    Therefore, by compactness, 
    there is $V \in \clopup(X)$ such that $x \in V$ and $V \cap \cl \min Y_d \subseteq \varphi(a)$.  
        Thus, $V \cap \min Y_d \subseteq \varphi(a)$, and so $x \in \rho(\varphi(a))$.

    ($\supseteq$)  By \cref{y and closures}, \begin{align*}
    \rho(\varphi(a)) \cap \min Y_d 
    &= \cl\bigcup \{V \in \clopup(X) \mid V \cap \min Y_d \subseteq \varphi(a)\} \cap \min Y_d \\
    &= \bigcup \{V \in \clopup(X) \mid V \cap \min Y_d \subseteq \varphi(a)\} \cap \min Y_d \subseteq \varphi(a).
    \end{align*}
    Therefore, since $\varphi(a)$ is closed, \[
        \rho(\varphi(a)) \cap \cl \min Y_d \subseteq \cl(\rho(\varphi(a)) \cap \min Y_d) \subseteq \varphi(a).
    \]
    Thus, for each $x \in \rho(\varphi(a))$, \[\upset x \cap \cl \min Y_d \subseteq \rho(\varphi(a)) \cap \cl \min Y_d \subseteq \varphi(a).\] Consequently, $x \in \varphi(ja)$ by \cref{Nj-2}.
\end{proof}

We conclude this section by describing the link between 
$\min Y_d$ and the localic part of~$N_{\ds}$.
\begin{proposition}
    The localic part $Y_{\ds}$ of $N_{\ds}$ is the soberification of $\min Y_d$.
\end{proposition}        

\begin{proof}
The localic part of any L-space is the space of points of the associated frame (see \cref{topology Y}). Thus, $Y_{\ds}$ is the space of points of $\Omega(\min Y_d)$, which is the soberification of $\min Y_d$ (see, e.g., \cite[p.~44]{Johnstone1982}).     
\end{proof}

\section{\texorpdfstring
{Compactness of the maximal $d$-spectrum}
{Compactness of the maximal d-spectrum}} \label{sec: compactness}

We now turn our attention to studying topological properties of $\min Y_d$. As we mentioned in the introduction, in \cite{Bhattacharjee2018} the second author only considered arithmetic frames with a \emph{unit}; that is, a compact dense element, where we recall (see the paragraph before \cref{Fnegneg = max X}) that an element $a \in L$ is \emph{dense} if $a^{**} = 1$. In this section, we characterize units in the language of Priestley spaces and compare the existence of a unit to compactness of $\min Y_d$. 

We start by characterizing compact subsets of $\min Y_d$ in terms of special Scott upsets of~$X$.

\begin{definition} \label{def: dinitial}
	Let $X$ be an arithmetic L-space. A subset $Z \subseteq X$ is called \emph{$d$-initial} if $Z \cap Y \subseteq \upset (Z \cap \min Y_d)$.
\end{definition}

\begin{lemma} \plabel{dHM lemma}
	Let $X$ be an arithmetic L-space and $Y$ its localic part.
	\begin{enumerate}
		        \item A Scott upset $F \subseteq X$ is $d$-initial iff $F = \upset(F \cap \min Y_d)$. \label[dHM lemma]{dhm0}
		\item If $F \subseteq X$ is a $d$-initial Scott upset, then $F \cap \min Y_d$ is compact. \label[dHM lemma]{dhm1}
		\item If $K \subseteq \min Y_d$ is compact then $\upset K$ is a $d$-initial Scott upset. \label[dHM lemma]{dhm2}
	\end{enumerate}
	\end{lemma}

\begin{proof}
    \ref{dhm0} The right-to-left implication is immediate. For the left-to-right implication, let $F$ be a $d$-initial Scott upset. Then $\min F \subseteq F \cap Y \subseteq \upset (F \cap \min Y_d)$, and hence $F = \upset \min F = \upset (F \cap \min Y_d)$. 
    
	\ref{dhm1} Suppose that $F \cap \min Y_d \subseteq \bigcup (U_i \cap \min Y_d)$ for a family $\{U_i\} \subseteq \clopup(X)$. Then $F \cap \min Y_d \subseteq \bigcup U_i$, and so $F \subseteq \bigcup U_i$ by \ref{dhm0}. Since $F$ is closed, it is compact, and hence $F \subseteq U_{i_1} \cup \dots \cup U_{i_n}$ for some $i_1, \dots, i_n$. Therefore,
	\[
		F \cap \min Y_d \subseteq (U_{i_1} \cap \min Y_d) \cup \dots \cup (U_{i_n} \cap \min Y_d),
	\]
	and hence $F \cap \min Y_d$ is compact.

	\ref{dhm2} Clearly, $\upset K$
  is $d$-initial. Thus, it suffices to show that $\upset K$ is a Scott upset. Since $\min \upset K = K \subseteq \min Y_d$, it is enough to show that $\upset K$ is closed. Let $x \notin \upset K$. Then $y \not\leq x$ for all $y \in K$. By Priestley separation, there is $U_y \in \clopup(X)$ such that $y \in U_y$ and $x \not \in U_y$. Therefore, $K \subseteq \bigcup (U_y \cap \min Y_d)$, so by compactness of $K$ there is $U \in \clopup(X)$ such that $\upset K \subseteq U$ and $x \notin U$. Thus, $\upset K$ is closed.
\end{proof}
As an immediate consequence, we have:
\begin{proposition} \plabel{prop: compact}
    Let $X$ be an arithmetic $L$-space and $K\subseteq \min Y_d$. The following are equivalent:
    \begin{enumerate}
        \item $K$ is compact. \label[prop: compact]{pc1}
        \item $\upset K$ is a $d$-initial Scott upset. \label[prop: compact]{pc2}
        \item There is a $d$-initial Scott upset $F\subseteq X$ such that $K=F\cap \min Y_d$. \label[prop: compact]{pc3}
     \end{enumerate}
\end{proposition}

\begin{theorem} \label{thm: compacts of min Yd}
	Let $X$ be an arithmetic L-space. There is a poset isomorphism between compact subsets of $\min Y_d$ and $d$-initial Scott upsets of $X$ \textup{(}both ordered by inclusion\textup{)}.
\end{theorem}

\begin{proof}
	Consider the maps $F \mapsto F \cap \min Y_d$ and $K \mapsto \upset K$, where $F \subseteq X$ is a $d$-initial Scott upset and $K \subseteq \min Y_d$ is compact. These are well defined by \cref{dHM lemma}, and are clearly order preserving. It suffices to show that these maps are inverses of each other. But $F = \upset (F \cap \min Y_d)$ by \cref{dhm0}, and it is easy to see that $K = \upset K \cap \min Y_d$, completing the proof.
\end{proof}

We recall (see, e.g., \cite[p.~25]{PicadoPultr2012}) that a frame is \emph{max-bounded} if each proper element is below a maximal element. 

\begin{proposition}
	Let $L$ be an arithmetic frame and $X$ its Priestley space. Then $L_d$ is max-bounded iff $N_d$ is $d$-initial.
\end{proposition}

\begin{proof}
	Since $L_d$ is max-bounded iff every proper $d$-upset is contained in a maximal $d$-upset (see \cref{rem: d upsets}), it suffices to show that the latter condition is equivalent to $N_d$ being $d$-initial.

	($\Rightarrow$) Let $y \in N_d \cap Y$. Then $y \in Y_d$ by \cref{Y_j = Y cap N_j}, so $X \setminus \downset y$ is a $d$-upset by \cref{Yd is d-element}. Hence, there exists a maximal $d$-upset $U$ such that $X \setminus \downset y \subseteq U$. By \cref{min Yd is max d-elements}, $U = X \setminus \downset y'$ for some $y' \in \min Y_d$. Thus, $X \setminus \downset y \subseteq X \setminus \downset y'$, which implies that $y' \in \downset y' \subseteq \downset y$. Therefore, $y' \leq y$, as required.

	($\Leftarrow$) Let $U$ be a proper $d$-upset. Then $U = \cl\dcore U \neq X$, so $Y_d \not \subseteq U$ by \cref{dU=X condition}. Hence, there is $y \in Y_d \setminus U \subseteq N_d$. Since $N_d$ is $d$-initial, there is $y' \in \min Y_d$ with $y' \leq y$. Therefore, $U \subseteq X \setminus \downset y'$, which is a maximal $d$-upset by \cref{min rmax gives maximal d-elements}.
\end{proof}

It is well known that if an arithmetic frame $L$ has a unit, then $L_d$ is max-bounded (see, e.g., \cite[before Prop.~3.3]{Bhattacharjee2018}). The following example shows that $L_d$ being max-bounded is a strictly weaker condition.

\begin{example} \label{example no unit}
	Let $L = \wp(\mathbb N)$. Then $L$ is an arithmetic frame, and $da = a$ for all $a \in L$ since $L$ is Boolean, so $L= L_d$. The maximal elements of $L_d$ are exactly the coatoms. Therefore, $L_d$ is max-bounded since it is atomic. However, $L_d$ does not contain a unit since the only dense element is $1$, which is not compact. 
\end{example}

We end this section by describing the Priestley analogue of having a unit and its relation to compactness of $\min Y_d$.

\begin{theorem} \plabel{min Yd compact}
	Let $L$ be an arithmetic frame and $X$ its L-space. The following are equivalent.
	\begin{enumerate}
		\item  There is a unit in $L$.
		\label[min Yd compact]{min Yd compact-1}
		\item There is a cofinal $U \in \clopsup(X)$. \label[min Yd compact]{min Yd compact-2}
        \item There is $U \in \clopsup(X)$ such that $Y_d \subseteq U$.
        \label[min Yd compact]{min Yd compact-2.5}
	\end{enumerate}
	The previous conditions imply the following equivalent conditions.
	\begin{enumerate}[resume]
		\item $\upset \min Y_d$ is a Scott upset. \label[min Yd compact]{min Yd compact-3}
		\item $\min Y_d$ is compact. \label[min Yd compact]{min Yd compact-4}
					\end{enumerate}
	If in addition $N_d$ is $d$-initial, then all five conditions are equivalent.
\end{theorem}

\begin{proof}
	\ref{min Yd compact-1}$\Leftrightarrow$\ref{min Yd compact-2} By \cref{rem: 2.8}, $a\in K(L)$ iff $\varphi(a)$ is a Scott upset. Moreover,  $a$ is dense iff
		$\max X \subseteq \varphi(a)$ by \cref{j=1 iff contains N_j,Nnegneg=max X}.

    \ref{min Yd compact-2}$\Leftrightarrow$\ref{min Yd compact-2.5}
     Suppose $U \in \clopsup(X)$ is cofinal and $y \in Y_d$. Then $\max \upset y \subseteq \max X \subseteq U$, so $y \in U$ by \cref{eqv rmax 3}. Therefore, $Y_d \subseteq U$. Conversely, suppose $Y_d \subseteq U$. Then $N_d = \cl Y_d \subseteq U$. But $N_d$ is cofinal by \cref{lem: dense = cofinal}, so $\max X \subseteq U$.

	\ref{min Yd compact-2.5}$\Rightarrow$\ref{min Yd compact-3} Since $N_d$ is inductive, $\upset (U \cap N_d)$ is a Scott upset. Thus, $\min \upset (U \cap N_d) \subseteq Y$, so 
	\[
		\min \upset (U \cap N_d) \subseteq Y \cap N_d = Y_d.
	\]
	We show that $\min \upset(U\cap N_d) = \min Y_d$.
    If $y \in \min \upset(U \cap N_d)$ and $y' \in Y_d$ such that $y' \leq y$, then $y' \in U \cap N_d$ because $Y_d \subseteq U$, so $y = y'$ since 
    \[
    	y \in \min \upset (U \cap N_d)=\min (U \cap N_d).
    \]
    Hence, $\min \upset(U \cap N_d) \subseteq \min Y_d$. Conversely, if $y  \in \min Y_d$ and $x \in \upset(U \cap N_d)$ with $x \leq y$, then there is $y' \in \min \upset (U \cap N_d) \subseteq Y_d$ with $y' \leq x \leq y$, so $y = y'$ since $y \in \min Y_d$. Consequently, $\min Y_d = \min \upset(U\cap N_d)$. Thus, $\upset \min Y_d = \upset(U \cap N_d)$, and hence $\upset \min Y_d$ is a Scott upset.
 
	\ref{min Yd compact-3}$\Leftrightarrow$\ref{min Yd compact-4} Since $\upset \min Y_d$ is $d$-initial, this follows from \cref{prop: compact}.

	Finally, suppose that $N_d$ is $d$-initial.

	\ref{min Yd compact-4}$\Rightarrow$\ref{min Yd compact-2.5} Since $\min Y_d$ is compact, the open cover $\{V \cap \min Y_d \mid V \in \clopsup(X)\}$ of $\min Y_d$ has a finite subcover, and since finite unions of clopen Scott upsets are clopen Scott upsets, there is $V \in \clopsup(X)$ such that $\min Y_d \subseteq V$. 
    By \cref{Y_j = Y cap N_j}, $Y_d = N_d \cap Y$. Thus, since $N_d$ is $d$-initial, $Y_d \subseteq \upset \min Y_d \subseteq V$.
\end{proof}

\begin{remark} 
    In \cref{example: no unit and compact} we will show that the assumption in \cref{min Yd compact} that $N_d$ is $d$-initial is necessary. In fact, we will see that $\min Y_d$ may be compact Hausdorff without $L$ having a unit.
\end{remark}

\section{\texorpdfstring
{Hausdorffness of the maximal $d$-spectrum}
{Hausdorffness of the maximal d-spectrum}}\label{Section 7}

In this final section, we give an example of an arithmetic frame $L$ with a unit such that $\max L_d$ is not Hausdorff, thus answering the question of \cite{Bhattacharjee2018} in the negative. In addition, we characterize exactly when $\max L_d$ is Hausdorff. 
Our characterization doesn't require that $L$ has a unit, only that $\max L_d$ is locally compact. 
Under this assumption, we show that $\max L_d$ is Hausdorff iff it is stably locally compact, a condition that plays an important role in domain theory (see \cite[Sec.~VI.6]{Compendium2003}). 

\begin{example} \label{main example}
Consider the Stone-Čech compactification of the natural numbers $\beta \mathbb N$. Partition the natural numbers in countably many countable subsets $\mathbb N = X_0 \cup X_1 \cup  X_2 \cup \dots$, where $X_i = \{x_{i,0},x_{i,1}, x_{i,2}, \dots\}$. Then for each $X_i$, $\cl X_i$ is a clopen set of $\beta \mathbb N$ homeomorphic to $\beta \mathbb N$ (see, e.g., \cite[p.~174]{Engelking1989}). Let $X_i^* = \cl (X_i) \cap \mathbb N^*$ and let $Y_\omega = \{y_0, y_1, y_2, \dots\} \cup \{\omega\}$ be the one-point compactification of a copy of the natural numbers. Consider now the disjoint union $X = \beta \mathbb N \sqcup Y_\omega$ and the partial order in \cref{fig:non hausdorff}, where $X_\omega^* = \mathbb N^* \setminus \bigcup_{n\in\mathbb N} X_n^*$.

\begin{figure}[H]
\centering
\begin{tikzpicture}[
	bullet/.style={fill,circle,inner sep=2pt},scale=.75]
	\begin{scope}[nodes=bullet]
		\node[label=below:$y_0$] (y0) at (3,-3) {};
		\foreach \x in {0,1,2} {
			\node[label={[label distance=-2mm]above:$x_{0,\x}$}] (\x) at (\x,0) {};
			\draw (y0) to (\x);
			};
	\end{scope}
		\path (2) -- node[auto=false]{\ldots} (4,0);
		\draw[very thick, |-|] (4, 0) to node[label=above:$X^*_0$] {} (6,0);
		\draw (y0) to (4,0);
		\draw (y0) to (6,0);
	\begin{scope}[shift={(7.5,0)}]
		\begin{scope}[nodes=bullet]
			\node[label=below:$y_1$] (y1) at (3,-3) {};
		\foreach \x in {0,1,2} {
			\node[label={[label distance=-2mm]above:$x_{1,\x}$}] (\x) at (\x,0) {};
			\draw (y1) to (\x);
			};
		\end{scope}
		\path (2) -- node[auto=false]{\ldots} (4,0);
		\draw[very thick, |-|] (4, 0) to node[label=above:$X^*_1$] {} (6,0);
		\draw (y1) to (4,0);
		\draw (y1) to (6,0);
	\end{scope}
	\path (14,-1.5) -- node[auto=false]{\huge\ldots} (16,-3);
	\node[label=below:$\omega$,bullet] (omega) at (15.5,0) {};
	\draw (y0) -- ($(y0)!.45!(omega)$);
	\draw ($(y0)!.825!(omega)$) -- (omega);
	\draw (y1) -- (omega);
	\draw[very thick, |-|] (14, 3) to node[label=above:$X_\omega^*$] {} (17,3);
	\draw (omega) -- (17,3);
	\draw (omega) -- (14,3);
	\end{tikzpicture}
	\caption{The Priestley space of an arithmetic frame whose maximal $d$-spectrum is not Hausdorff.}\label{fig:non hausdorff}
	\end{figure}
\end{example}

Our goal is to show that $X$ is an arithmetic L-space such that $\min Y_d$ is not Hausdorff. We have several things to verify.

\begin{claim}
	$X$ is a Priestley space.
\end{claim}
\begin{proof}
$X$ is a Stone space since it is the disjoint union of two Stone spaces. It remains to be shown that $X$ satisfies the Priestley separation axiom. For $x \in \beta \mathbb N$ and $x' \in X$ with $x\not\leq x'$, finding a clopen upset containing $x$ and missing $x'$ is easy since $\beta \mathbb N$ is a clopen upset of $X$.

Let $y \in Y_\omega$ and $x \in X$ with $y \not \leq x$. Then $x \notin \upset \omega$, so $x \in \downset \cl X_i$ for some $i$ such that $y \neq y_i$. Consider $U = X \setminus \downset \cl X_i$. Since $\cl X_i$ is clopen, so is $\downset \cl X_i$. Thus, $U$ is a clopen upset separating $y$ from $x$.
\end{proof}
\begin{claim} \label{claim: Lspace}
	$X$ is an L-space.
\end{claim}
\begin{proof}
	    It is sufficient to show that $\cl U$ is a clopen upset for each open upset $U \subseteq X$. Let $U$ be an open upset. Then $U \cap Y_\omega$ is open, and it is either empty or it must contain $\omega$. In both cases, $\cl(U) \cap Y_\omega = U \cap Y_\omega$, so $\cl U = U \cup \cl (U \cap \beta\mathbb N)$, which is clearly a clopen upset.
    \end{proof}

\begin{claim} \label{claim: loc part}
	The localic part of $X$ is $Y = Y_\omega \cup \mathbb N$.
\end{claim}
\begin{proof}
	We have $Y \cap \mathbb N^* = \varnothing$ since $\downset x \cap \beta \mathbb N = \{x\}$ is not open for all $x \in \mathbb N^*$, so $Y \subseteq Y_\omega \cup \mathbb N$. For the converse, if $y \in Y_\omega \setminus \{\omega\}$ then $\downset y = \{y\}$ is clopen. Also, $\downset \omega = Y_\omega$ is clopen, so $Y_\omega \subseteq Y$. For $x \in X_i$, we have $\downset x = \{x,y_i\}$ is clopen.
\end{proof}

\begin{claim} \plabel{claim: clopsups}
	Let $U \subseteq X$ be an upset. Then $U \in \clopsup(X)$ iff one of the following two conditions holds.
	\begin{enumerate}
			\item $U$ is a finite subset of $\mathbb N$. \label[claim: clopsups]{cond 1}
			\item $U \cap Y_\omega$ is cofinite, and $y_i \notin U$ implies $\cl(X_i) \cap U$ is a finite subset of $X_i$. \label[claim: clopsups]{cond 2}
		\end{enumerate}
\end{claim}
\begin{proof}
	($\Rightarrow$) Suppose $U \in \clopsup(X)$ and $U \cap Y_\omega$ is not cofinite. Since $U \cap Y_\omega$ is clopen and not cofinite, $\omega \not \in U$. Hence, $U \cap Y_\omega = \varnothing$ since it is an upset. Therefore, $U \subseteq \beta\mathbb N = \cl \mathbb N$. By \cref{Scott upset lemma}, $U \subseteq \mathbb N$, and by compactness $U$ is finite. Suppose now that $U \cap Y_\omega$ is cofinite. If $y_i \notin U$, then $X_i^* \cap U = \varnothing$, since otherwise $U$ can't be Scott upset because $\min U \not\subseteq Y$ (see \cref{claim: Lspace}). Therefore, $U \cap \cl(X_i) \subseteq X_i$, and it has to be finite since it is compact.

	($\Leftarrow$) If $U$ is a finite subset of $\mathbb N$, then $U \in \clopup(X)$ and $U \subseteq Y$ by \cref{claim: loc part}, so $U \in \clopsup(X)$. Now suppose \ref{cond 2} holds. Then $\omega \in U$ and $y_i \notin U$ implies $X_i^* \cap U = \varnothing$. Hence, $\min U \subseteq Y_\omega \cup \mathbb N = Y$, so it suffices to show that $U$ is clopen. Since $U \cap Y_\omega$ is cofinite it is clopen in $Y_\omega$. Moreover,
	\[U \cap \beta \mathbb N = \bigcup\{\cl(X_i) \cap U \mid y_i \not \in U\} \cup \bigcup\{\cl(X_i) \mid y_i \in U\} \cup X_\omega^*.\]
	By \ref{cond 2}, $\bigcup\{\cl(X_i) \cap U \mid y_i \not \in U\} = \bigcup\{X_i \cap U \mid y_i \not \in U\} \subseteq \mathbb N$ is finite and hence clopen. Also, $\beta \mathbb N \setminus (\bigcup\{\cl(X_i) \mid y_i \in U\} \cup X_\omega^*)$ is clopen since only finitely many $y_i \notin U$. Therefore, $\bigcup\{\cl(X_i) \mid y_i \in U\} \cup X_\omega^*$ is clopen. Thus, $U \cap \beta \mathbb N$ is clopen, and so $U = (U \cap Y_\omega) \cup (U \cap \beta \mathbb N)$ is clopen.
\end{proof}

\begin{claim}
	$X$ is an algebraic L-space.
\end{claim}
\begin{proof}
	Suppose $U \in \clopup(X)$. We need to show $U \subseteq \cl(\core U)$, so suppose $x \in U$. We consider three cases.
	\begin{enumerate}[(i)]
		\item
		If $x \in Y_\omega$ then $\omega \in \upset x \subseteq U$, so $U \cap Y_\omega$ is cofinite. Then 
		\[
			x \in \upset(U \cap Y_\omega) \in \clopsup(X)
			\]
			by \cref{cond 2}, and $\upset (U \cap Y_\omega) \subseteq \core U$.
		\item If $x \in \mathbb N$, then $\upset x = \{x\} \in \clopsup(X)$  by \cref{cond 1}, and ${\upset x \subseteq \core U}$.
		\item Suppose $x \in \mathbb N^*$. Since $U$ is clopen in $X$, $U \cap \beta \mathbb N$ is clopen, and therefore 
		\[
			\cl (U \cap \mathbb N) = U \cap \beta \mathbb N.
		\]
		Hence, $x \in \cl (U \cap \mathbb N)$. Suppose now that $V$ is a clopen neighborhood of $x$. Then $V \cap (U \cap \mathbb N) \neq \varnothing$, but $U \cap \mathbb N = \bigcup\{\{n\} \mid n \in \mathbb N \cap U\} \subseteq \core U$ since $\{n\} \in \clopsup(X)$ by \cref{cond 1}. Therefore, $x \in \cl \core U$. \qedhere
	\end{enumerate}
\end{proof}

\begin{claim} \label{claim: ari l-space}
	$X$ is an arithmetic L-space.
	\end{claim}
\begin{proof}
	It suffices to show that $U \cap V \in \clopsup(X)$ for $U,V \in \clopsup(X)$ (see \cref{def arithemitc}), so suppose $U,V \in \clopsup(X)$. Then $U$ and $V$ satisfy one of the two conditions of \cref{claim: clopsups}. If either $U$ or $V$ is a finite subset of $\mathbb N$, then so is their intersection. Suppose $U$ and $V$ both satisfy \cref{cond 2}. Since a finite intersection of cofinite sets is cofinite, $U \cap V \cap Y_\omega$ is cofinite. If $y_i \notin U \cap V$, then either $y_i \notin U$ or $y_i \notin V$. Without loss of generality we may assume the former. Then $\cl(X_i) \cap U \cap V \subseteq \cl(X_i) \cap U$ is a finite subset of $X_i$. Thus, \cref{cond 2} holds for $U \cap V$, and so $U \cap V \in \clopsup(X)$.
	\end{proof}

\begin{claim}
	$\min Y_d = Y_\omega \setminus \{\omega\}$.
\end{claim}
\begin{proof}
	Observe that for each $y \in Y = \mathbb N \cup Y_\omega$, there is $x \in \max X$ with $\{y\} = \max(\downset x \cap Y)$. Therefore, $Y_d = \mathbb N \cup Y_\omega$ by \cref{eqv rmax 4}. Consequently, $\min Y_d = Y_\omega \setminus\{\omega\}$.
\end{proof}

\begin{claim} \label{claim: not hausdorff}
	$\min Y_d$ is not Hausdorff.
\end{claim}
\begin{proof}
	$\clopsup(X)$ forms a basis of $\min Y_d$ because $U \cap Y = \core U \cap Y$ for each $U \in \clopup(X)$. Consequently, it follows from \cref{cond 2} that $\min Y_d$ is equipped with the cofinite topology, which is not Hausdorff since $\min Y_d$ is infinite.
		\end{proof}

\cref{claim: ari l-space,claim: not hausdorff} yield the following: 
\begin{theorem}
	There exist arithmetic L-spaces $X$ such that $\min Y_d$ is not Hausdorff.
\end{theorem}

As promised at the end of \cref{sec: compactness}, we now demonstrate that $\min Y_d$ can be compact Hausdorff without $L$ having a unit.

\begin{example} \label{example: no unit and compact}
    Redefine the order in the space $X$ of
	 \cref{main example}
	 as in \cref{fig:non min}.
\begin{figure}[H]
\centering
\begin{tikzpicture}[
	bullet/.style={fill,circle,inner sep=2pt},scale=.75]
	\begin{scope}[nodes=bullet]
		\node[label=below:$y_0$] (y0) at (3,-3) {};
		\foreach \x in {0,1,2} {
			\node[label={[label distance=-2mm]above:$x_{0,\x}$}]  (\x) at ($(\x+2,0)$) {};
			\draw (y0) to (\x);
			};
	\end{scope}
		\path (2) -- node[auto=false]{\ldots} (6,0);
		\draw[very thick, |-|] (6, 0) to node[label=above:$X^*_0$] {} (8,0);
		\draw (y0) to (6,0);
		\draw (y0) to (8,0);
		
	\begin{scope}[shift={(5,-2)}]
		\begin{scope}[nodes=bullet]
		\node[label=below:$y_1$] (y1) at (3,-3) {};
		\foreach \x in {0,1,2} {
			\node[label={[label distance=-2mm]above:$x_{1,\x}$}] (\x) at ($(\x+2,0)$) {};
			\draw (y1) to (\x);
			};
		\end{scope}
		\path (2) -- node[auto=false]{\ldots} (6,0);
		\draw[very thick, |-|] (6, 0) to node[label=above:$X^*_1$] {} (8,0);
		\draw (y1) to (6,0);
		\draw (y1) to (8,0);
			\end{scope}
	\begin{scope}[shift={(11,-4.4)}]
	\node[label=below:$\omega$,bullet] (omega) at (3,-3) {};
	\draw[very thick, |-|] (5, 0) to node[label=above:$X_\omega^*$] {} (9,0);
	\draw (omega) -- (5,0);
	\draw (omega) -- (9,0);
	\end{scope}
	\draw (y1) -- (y0);
    \draw (y1) -- ($(y1)!.5!(omega)$);
	\path (y1) -- node[pos=.75, sloped, fill=white, outer sep=1em]{\huge\ldots} (omega);
	\end{tikzpicture}
	\caption{The Priestley space of an arithmetic frame without a unit whose maximal $d$-spectrum is compact Hausdorff}\label{fig:non min}
	\end{figure}
	\noindent 
        A similar reasoning to the above yields that $X$ is an arithmetic L-space and its localic part is $Y \coloneqq \mathbb N \cup (Y_\omega \setminus \{\omega\})$. Observe that $Y_d = Y$ and $\min Y_d = \varnothing$, so it is trivially compact Hausdorff. However, $X$ has no cofinal clopen Scott upset since $X_{\omega}^* \subseteq \max X$ and $\downset X_{\omega}^* \cap Y = \varnothing$. Consequently, $X$ is the L-space of an arithmetic frame without units.
\end{example}

To characterize when $\min Y_d$ is Hausdorff, we recall (see, e.g., \cite[Def.~VI-6.7]{Compendium2003}) that a topological space $X$ is \emph{coherent} if the binary intersection of compact saturated sets is compact. The space $X$ is \emph{stably locally compact} if it is sober, locally compact, and coherent. A stably locally compact space is \emph{stably compact} if it is compact, and \emph{spectral} if in addition compact open sets form a basis.
It is well known (see, e.g., \cite[p.~75]{Johnstone1982}) that a spectral space is Hausdorff iff it is $T_1$.
The next lemma generalizes this result to 
stably locally compact spaces.

\begin{lemma} \label{lemma: min Yd stablylc implies t2}
	If $X$ is stably locally compact, then $X$ is Hausdorff iff $X$ is $T_1$.
\end{lemma}
\begin{proof}
	We only need to show the right-to-left implication. Suppose $x, y \in X$ are distinct. Let $\mathcal K_x = \{K \subseteq X \mid K \mbox{ is a compact saturated neighborhood of }x\}$ and define $\mathcal K_y$ similarly.  It suffices to show that there exist $K_x \in \mathcal K_x$ and $K_y \in \mathcal K_y$ such that $K_x \cap K_y = \varnothing$. 
        Since $X$ is $T_1$ and locally compact, 
    for each $z \in X$ distinct from $x$, there is a compact saturated neighborhood $K$ of $x$ missing $z$. Therefore, $\bigcap \mathcal K_x = \{x\}$, and similarly ${\bigcap \mathcal K_y = \{y\}}$. Consequently,
    $\bigcap \mathcal K_x \cap \bigcap \mathcal K_y = \varnothing$.
    By \cite[Lem.~VI-6.4]{Compendium2003}, there exists a finite $\mathcal K \subseteq \mathcal K_x \cup \mathcal K_y$ such that $\bigcap \mathcal K = \varnothing$. Since $X$ is stably locally compact, $\mathcal K_x$ and $\mathcal K_y$ are directed. Therefore, there are $K_x \in \mathcal K_x$ and $K_y \in \mathcal K_y$ such that $K_x \cap K_y = \varnothing$. Thus, $X$ is Hausdorff. 
\end{proof}

\begin{theorem} \label{thm: Hausdorff}
    Let $X$ be an arithmetic L-space such that $\min Y_d$ is locally compact. Then $\min Y_d$ is Hausdorff iff $\min Y_d$ is stably locally compact.
\end{theorem}
\begin{proof}
	First suppose that $X$ is Hausdorff. Then $X$ is sober (see, e.g., \cite[p.~43]{Johnstone1982}). Let $K,M\subseteq X$ be compact saturated. Since $X$ is Hausdorff, $K,M$ are closed, so $K\cap M$ is closed. Since it is a closed subset of a compact set, it must be compact. Thus, $X$ is stably locally compact. Conversely, since $\min Y_d$ is $T_1$ by \cref{fact: t1}, $\min Y_d$ is Hausdorff by \cref{lemma: min Yd stablylc implies t2}.
 \end{proof}

\begin{corollary} 
	Let $L$ be an arithmetic frame with a unit and $X$ its L-space.     \begin{enumerate}
        \item $\min Y_d$ is Hausdorff iff $\min Y_d$ is stably locally compact. 
        \item $\max L_d$ is Hausdorff iff $\max L_d$ is stably locally compact.
    \end{enumerate}
\end{corollary}
\begin{proof} 
    We only prove (1) as (2) follows from (1) and \cref{thm: min Yd = max Ld}. 
    Since $L$ has a unit, $\min Y_d$ is compact by 
    \cref{min Yd compact}. 
    First suppose that $\min Y_d$ is Hausdorff. Then
        $\min Y_d$ is compact Hausdorff, and hence $\min Y_d$ is stably locally compact. Conversely, if $\min Y_d$ is stably locally compact, then \cref{thm: Hausdorff} applies, and hence $\min Y_d$ is Hausdorff. 
        \end{proof}

We conclude the paper with several interesting open problems:
\begin{itemize}
    \item It remains open
whether \cref{thm: Hausdorff} can be reformulated as an equivalence between sobriety and Hausdorffness. Note that in \cref{main example}, $\min Y_d$ is locally compact and coherent, but it fails to be Hausdorff solely because it is not sober. 
    \item     It also remains open whether $\min Y_d$ is always locally compact (and/or coherent). In the absence of sobriety, local compactness of $\min Y_d$ is not equivalent to 
        $\Omega(\min Y_d)$ being a continuous frame (see, e.g., \cite[p.~310]{Johnstone1982}). This disparity emphasizes the importance of sobriety in these considerations. Indeed, as noted above, it is
        plausible that in this setting sobriety alone implies Hausdorffness.
        \item Resolving the above questions requires developing a general method for identifying which topological spaces can be realized as $\min Y_d$. While each Stone space can be realized as such, it remains open whether the same can be said about each compact Hausdorff space (we note that it follows from \cite{HenriksenVermeerWoods1987} that each compact Hausdorff quasi F-space can be realized this way). 
\end{itemize}

\bibliographystyle{alpha-init}
\bibliography{ref}

\end{document}